\documentclass{amsart}
\usepackage{amssymb}
\usepackage{amsmath}
\usepackage{subfig}
\usepackage{epsfig}
\usepackage{graphicx}
\usepackage{epic}
\usepackage[mathscr]{eucal}
\usepackage{xspace}
\usepackage[latin1]{inputenc}

\newcommand{\HH}{\mathcal H}

\newcommand{\mm}{\mathfrak{m}}
\renewcommand{\SS}{\mathbb{S}}

\newcommand{\R}{\mathbb{R}}
\newcommand{\Q}{\mathbb{Q}}

\newcommand{\N}{\mathbb{N}}

\newcommand{\FF}{\mathcal{F}}

\newcommand{\KK}{\mathcal{K}}

\newcommand{\CS}{\mathcal{S}}

\newcommand{\po}{\partial}

\newcommand{\wto}{\rightharpoonup}
\newcommand{\ve}{\varepsilon}

\newcommand{\la}{\langle}
\newcommand{\ra}{\rangle}

\newcommand{\loc}{{\text{loc}}}
\newcommand{\X}{\times}

\renewcommand{\d}{\delta}
\renewcommand{\l}{\lambda}

\renewcommand{\a}{\alpha}

\newcommand{\s}{\sigma}
\newcommand{\g}{\gamma}

\newcommand{\z}{\zeta}
\newcommand{\Id}{\operatorname{Id}}
\newcommand{\sol}{\operatorname{sol}}
\newcommand{\pot}{\operatorname{pot}}
\newcommand{\mv}[1]{\overline{#1}}

\newcommand{\Om}{\Omega}

\newcommand{\supp}{\text{\rm supp}\,}
\newcommand{\re}{\mathbb{R}}
\newcommand{\M}{{\mathcal M}}

\renewcommand{\div}{\text{\rm div}\,}

\renewcommand{\subset}{\subseteq}

\newcommand{\Nu}{\mathcal{V}}

\newcommand{\BUC}{\operatorname{BUC}}

\newcommand{\AP}{\operatorname{AP}}

\newcommand{\WAP}{\operatorname{WAP}}

\newcommand{\util}[1]{\underline{#1}}
\newcommand{\DD}{{\mathcal D}}
\newcommand{\BB}{{\mathcal B}}
\renewcommand{\AA}{{\mathcal A}}

\newcommand{\medint}{{\mbox{\vrule height3.5pt depth-2.8pt
          width4pt}\mkern-13mu\int\nolimits}}
\newcommand{\Medint}{\mkern12mu\mbox{\vrule height4pt
         depth-3.2pt
          width5pt}\mkern-16.5mu\int\nolimits}

\renewcommand{\supp}{\operatorname{supp}}

\theoremstyle{plain}
\newtheorem{theorem}{Theorem}[section]

\newtheorem{lemma}[theorem]{Lemma}

\theoremstyle{definition}
\newtheorem{definition}[theorem]{Definition}
\newtheorem{remark}[theorem]{Remark}
\theoremstyle{remark}

\numberwithin{equation}{section}

\begin{document}
\title[Stefan Problem]
{Homogenization of a generalized  Stefan Problem\\ in the context of ergodic algebras}
\author{Hermano Frid}
\address{Instituto de Matem\'atica Pura e Aplicada - IMPA\\ Estrada Dona Castorina, 110\\
Rio de Janeiro, RJ, 22460-320, Brazil}
\email{hermano@impa.br}

\author{Jean Silva}
\address{Instituto de Matem\'atica, Universidade Federal do Rio de Janeiro}
\email{jean@im.ufrj.br}

\author{Henrique Versieux}
\address{Instituto de Matem\'atica, Universidade Federal do Rio de Janeiro}
\email{henrique@im.ufrj.br}

\date{}

\begin{abstract}
We address the deterministic homogenization, in the general context of ergodic algebras, of a doubly nonlinear problem which
generalizes the well known Stefan model, and includes the classical porous medium equation. It may be represented by the differential  inclusion, for a real-valued function $u(x,t)$,  
$$
 \frac{\partial}{\partial t}\partial_u \Psi(x/\ve,x,u)-\nabla_x\cdot \nabla_\eta\psi( x/\ve,x,t,u,\nabla u) \ni  f(x/\ve,x,t, u), 
$$
on a bounded domain $\Om\subset \R^n$, $t\in(0,T)$, together with initial-boundary conditions, where  $\Psi(z,x,\cdot)$ is strictly convex and $\psi(z,x,t,u,\cdot)$ is a $C^1$ convex function, both with quadratic growth,
satisfying some additional technical hypotheses. As functions of the oscillatory variable, $\Psi(\cdot,x,u),\psi(\cdot,x,t,u,\eta)$ and  $f(\cdot,x,t,u)$ belong to the generalized Besicovitch space $\BB^2$ associated with an arbitrary ergodic algebra $\AA$. The periodic case was addressed by Visintin (2007), based on the two-scale convergence technique. Visintin's analysis for the periodic case relies heavily on the possibility of reducing two-scale convergence to the usual $L^2$ convergence in the cartesian product $\Pi\X\R^n$, where $\Pi$ is the periodic cell. This reduction is no longer possible in the case of a general ergodic algebra.   To overcome this difficulty, we make essential use of the concept of two-scale Young measures for algebras with mean value, associated with  bounded  sequences in $L^2$. 

\end{abstract}

\maketitle

\section{Introduction} \label{S:1}

In this paper we are concerned with the homogenization, in the context of a general ergodic algebra $\AA$ on $\R^n$, of the following initial-boundary value problem, on $\Om\X(0,T)$, with
$\Om\subset\R^n$ a bounded open set,
 \begin{equation}\label{e0.1}
 \begin{aligned}
 &\po_t w_\ve -\nabla\cdot \a(\frac{x}{\ve},x,t, u_\ve,\nabla u_\ve)=f(\frac{x}{\ve}, x, u_\ve),\\
 & w_\ve(x,t)\in \partial \Psi(\frac{x}{\ve},x,u_\ve),\qquad \text{a.e.\ in $\Om\X(0,T)$},\\
 &w_\ve(x,0)=w_0(\frac{x}{\ve},x),\qquad x\in\Om,\\
 & u_\ve=0,\qquad \text{in $\po\Om\X(0,T)$}.
 \end{aligned}
\end{equation}
Here $\Psi(z,x,\cdot)$ is a strictly convex function for a.e.\ $(z,x)\in\R^n\X\Om$, $\po\Psi(z,x,v)$ denotes the subdifferential of $\Psi$ with respect to $v$,   and for all $v\in\R$, $\Psi(\cdot,\cdot,v) \in \BB^2(\R^n;L(\Om))$, where $\BB^2$ denotes the generalized Besicovitch space associated with the ergodic algebra $\AA$, with topology provided by the semi-norm given as the mean value of the square of the absolute value. Also, $\a(z,x,t,v,\eta)=\nabla_\eta\psi(z,x,t,v,\eta)$, where $\psi(z,x,t,v,\cdot)$ is $C^1$ and convex in $\R^n$, for all $v\in\R$ and a.e.\ $(z,x,t)\in\R^n\X\Om\X(0,T)$. In particular,
$\a(z,x,t,v,\cdot)$ is a montone operator which is also assumed to satisfy a coercivity condition such as $\a(z,x,t,v,\eta)\cdot\eta\ge c_0|\eta|^2+ h(x,t)$, with $c_0>0$ and $h\in L^1(\Om\X(0,T))$, as usual. Further, $\a(\cdot,\cdot,\cdot,v,\eta) \in\BB^2(\R^n;L^2(\Om\X(0,T))$ and $f(\cdot,\cdot,v)\in\BB^2(\R^n; L^2(\Om))$, for all $v\in\R$ and all $\eta\in\R^n$. Additional technical conditions are assumed over $\Psi,\psi,\a,f$, see {\bf($\mathbf{\Psi1}$)}--{\bf($\mathbf{\Psi4}$)}, {\bf($\mathbf{\psi1}$)}--{\bf($\mathbf{\psi4}$)}, {\bf($\mathbf{\a1}$)}--{\bf($\mathbf{\a5}$)}, in Section~\ref{S:4}, and {\bf($\mathbf{f1}$)}--{\bf($\mathbf{f3}$)} in Section~\ref{S:5}, below. The corresponding case where $\AA$ is the algebra of the continuous periodic functions on $\R^n$ was addressed in the pioneering paper by Visintin,  \cite{MR2349874}, and the main point of the present paper is to extend the results in  \cite{MR2349874}  to the general case where $\AA$ is an arbitrary ergodic algebra. The latter is a concept introduced by Zhikov and Krivenko in \cite{MR704444}, which abstracts the properties satisfied by realizations of continuous functions in general compact topological spaces, endowed with a probability measure $\mu$, under the action of a continuous dynamical system for which $\mu$ is invariant. It includes the  almost periodic functions, introduced by Bohr,  as well as the more general weak almost periodic functions, introduced by Eberlein, which strictly include the Fourier-Stieltjes transforms (see, e.g.,  \cite{MR0068029}, \cite{MR0036455}, and the discussion in Section~\ref{S:2}, below). The main difficulty in achieving this extension is concerned with the most challenging nonlinearity in this problem, represented by the subdifferential of the strictly convex function $\Psi$.  In \cite{MR2349874}, Visintin uses the method of the two-scale convergence  introduced by Nguetseng in  \cite{MR990867}, and further developed by Allaire in \cite{MR1185639}. To treat the main nonlinearity, just mentioned, his analysis is based on the fact that, in the periodic case, it is possible to introduce certain maps $S_\ve:\Pi\X\R^n\to \R^n$, where $\Pi$ denotes the periodic cell in $\R^n$,  such that,  $S_\ve(y,x)\to x$, as $\ve\to0$,  in $\Pi\X\R^n$, uniformly with respect to $y\in\Pi$, and, for a large class of measurable functions $f:\R^n\X\R^n\to \R$, $f(y,x)$ periodic in $y$, with period cell $\Pi$,  one has $\int_{\R^n}f(x/\ve,x)\,dx=\int_{\Pi\X\R^n}f(y, S_\ve(y,x))\,dy\,dx$, which allows to reduce the two-scale convergence to the usual  convergence in  Lebesgue spaces. This reduction is no longer possible in the case where the oscillatory coefficients belong to general ergodic algebras. Instead of using only the two-scale convergence, we make essential use of the two-scale Young measures, 
introduced in the periodic case by E \cite{MR1151269}, whose basic existence result  for algebras with mean value was established in \cite{MR2498566} (see also \cite{MR2481061}).  The main existence result for two-scale convergence in algebras with mean value was obtained by Casado-Diaz and Gayte in \cite{MR1899822}, while the corresponding result for stochastic homogenization was obtained by Bourgeat, Mikeli\'c and Wright, in \cite{MR1301450}. 

In general, for problem \eqref{e0.1}, the a priori bounds available give boundedness of $u_\ve$ in $L^2(0,T;H_0^1(\Om))$ and of $w_\ve$ in $L^2(0,T;L^2(\Om))\cap H^1(0,T;H^{-1}(\Om))$, uniformly in $\ve>0$. We apply the theorem on the existence of two-scales Young measures for bounded sequences in $L^2$ for the sequence $u_\ve$. Since $\Psi(z,x,u)$ has quadratic growth in $u$, and the representation formula for the two-scale Young measures, in principle, only applies to functions with strictly sub-quadratic growth, the difficulty to be overcome here is to show that this representation formula may be particularly extended to apply to $\Psi$.  Once this is done, we then show that the Young measures reduce to 
Dirac measures concentrated at the weak limit $u(x,t)$ of some weakly converging subsequence of $u_\ve$, which gives the strong convergence of this subsequence of $u_\ve$
to $u$, in $L^p(\Om\X(0,T)$, for any $1\le p<2$.   
  
The general equation in \eqref{e0.1} generalizes classical models in porous medium equation (see, e.g., \cite{MR877986,MR2286292}) and the Stefan problem  
(see, e.g., \cite{MR0125341,MR0141895,MR0181840,MR0227625,MR0487015,MR687720,MR670933,MR1423808} and references therein). Homogenization of the  Stefan  problem, in the periodic setting, was first studied by Damlamian, see \cite{MR613313,MR633807}.  For a discussion about the physical background leading to \eqref{e0.1} we refer to the introduction of Visintin's paper cited above. Existence of solution for the general problem \eqref{e0.1} is also addressed in  \cite{MR2349874}. Uniqueness is not known in general. However, uniqueness is indeed known in the case where the equation in \eqref{e0.1} has the simpler form
\begin{equation}\label{e0.2}
\po_tw_\ve-\nabla\cdot (K(x/\ve,x)\cdot\nabla u_\ve)=0,
\end{equation}
where $K(x/\ve,x)$ is a measurable $n\X n$-matrix-valued function of $x\in\Om$, satisfying an ellipticity condition of the form
\begin{equation}\label{e0.3}
\g_0|\xi|^2\le [K(x/\ve,x)\cdot\xi]\cdot \xi\le \g_1|\xi|^2, \quad \text{for all $\xi\in\R^n$ and a.e.\ $x\in\Om$},
\end{equation}
for some constants $\g_1>\g_2>0$.  This includes the classical models for the porous medium equation and for the Stefan problem, mentioned above.  

Also, for \eqref{e0.2}, the homogenized problem can be explicitly obtained, and has the form
\begin{equation}\label{e0.4}
 \begin{aligned}
&\po_t w-\nabla\cdot (K_0(x)\cdot\nabla u)=0, \\
 & w(x,t)\in \partial \mv{\Psi}(x,u),\qquad \text{a.e.\ in $\Om\X(0,T)$},\\
 &w(x,0)=\mv{w_0}(x),\qquad x\in\Om,\\
 & u=0,\qquad \text{in $\po\Om\X(0,T)$},
 \end{aligned}
\end{equation}
where $\mv{\Psi}(x,u)$ and $\mv{w_0}(x)$ are the mean values of $\Psi(z,x,u)$ and $w_0(z,x)$ with respect to the variable $z$, and $K_0(x)$ is an $n\X n$-matrix valued measurable function satisfying an ellipticity condition similar to \eqref{e0.3}, and is easily obtained by following the classical procedure for homogenizing linear elliptic and parabolic equations. We remark that the uniqueness result also applies to the homogenized problem, so that, in this case we get a complete characterization for the homogenization limit.   We will discuss this point in detail in Section~\ref{S:6}.

In the general case, our main result gives an homogenized problem of the form 
\begin{equation}\label{e0.5}
\begin{aligned}
& w_t-\nabla\cdot \mv{q}= \mv{f}\left(x,t,u\right), \\
& w(x,t)\in \partial \mv{\Psi}\left(x,u(x,t)\right),\\
&\mv{q}\in\partial{\psi}_0\left (x,t,u(x,t), \nabla u(x,t)\right),\quad \text{for a.e. $(x,t)\in \Omega\X(0,T)$},\\
&w(x,0)=\mv w_0(x), \quad x\in \Omega,\quad u=0,\quad\text{in $\partial\Omega\X(0,T)$},
\end{aligned}
\end{equation}
where,  $\mv{\Psi}(x,u), \mv f,  \mv w_0,   \mv{\psi_0}$ are the mean values of $\Psi(\cdot,x,u), f(\cdot,x,t,u), w_0(\cdot,x), \psi_0(\cdot,x,t, u,\eta)$ for each $(x,t,u,\eta)$ in  
$\Omega\X(0,T)\X\R\X\R^n$,  $\partial \mv{\Psi}, \partial\mv{\psi_0}$ represent the subdifferentials of the convex functions $\mv{\Psi}(x, \cdot), \mv{\psi_0}(x,u,t,\cdot)$ for each $(x,t)\in \Omega\X(0,T)$, $u\in\R$, and  the function $\psi_0(x,t,u,\eta)$ is obtained through the minimization of the convex integral functional with integrand $\psi(z,x,t, u(x,t), -\nabla u(x,t)-\eta(z))$, where  $\eta(z)$ runs along the potential fields with mean zero in the compactification $\KK$ of $\R^n$ induced by the algebra $\AA$ (see Sections~\ref{S:4} and \ref{S:5}).   This is the analogue of the homogenized problem obtained by Visintin in the periodic case, in \cite{MR2349874}.

We remark that instead of the null Dirichlet boundary condition appearing in \eqref{e0.1}, we could as well, without any relevant change in our analysis, consider a more general homogeneous boundary condition, as in \cite{MR2349874}, say, imposing the null Dirichlet condition in a part of the boundary $\Gamma_0$, and the zero flux Neumann condition on the complementary part $\po\Om\setminus\Gamma_0$. Since here we are mainly concerned with the homogenization problem, just for the neatness of the exposition,  we restrict ourselves to the simplest case of the null Dirichlet condition.  

This paper is organized as follows. In Section~\ref{S:2}, we make a brief review of the main facts about ergodic algebras that are going to be used in the subsequent sections.
In Section~\ref{S:3}, we recall the main existence result for two-scale Young measures in algebras with mean value. We make the statement in a more general fashion, suitable for bounded sequences of functions in $L^p$, for any $1<p\le\infty$, and we also show how the proof given originally in \cite{MR2481061} can be easily adapted to handle this more general case. Also in the same section, we recall the definition of two-scale convergence and show how the general theorem on the existence of two-scale Young measures can be applied to imply the main existence result for two-scale convergence, in algebras with mean value, originally obtained in \cite{MR1899822}; we actually establish a more general new result of interest in its own ({\em cf.} Theorem~\ref{T:3.2}). In Section~\ref{S:4}, we prove the two main results that serve as essential tools in our homogenization analysis. The first one, Theorem~\ref{T:4.1}, shows how to handle the more challenging nonlinearity, represented by the subdifferential of the strictly convex function $\Psi(z,x,\cdot)$. The second one, Theorem~\ref{T:4.2}, shows how to address the second nonlinearity represented by the monotone operator $\a(z,x,t,u,\cdot)$. These results are combined in Section~\ref{S:5} to establish our main homogenization result for the problem \eqref{e0.1}. Finally, in Section~\ref{S:6}, we discuss the special case where $\a(z,x,t,u,\eta)=K(z,x,t,u)\cdot\eta$, for some positive definite matrix-valued function $K(z,x,t,u)$, and review, in particular, the case where $K$ only depends on $(z,x)$, for which uniqueness is available for both the homogenizing and the homogenized problems.

\section{Ergodic Algebras}\label{S:2}

In this section we recall the basic facts concerning algebras with mean values and, in particular, ergodic algebras. To begin with, we recall
the notion of mean value for functions defined in $\re^n$.

\begin{definition}\label{D:3} Let $g\in L_\loc^1(\R^n)$. A number $\overline{g}$ is called the {\em mean value of $g$} if
\begin{equation}\label{e1.2}
\lim_{\ve \to0} \int_Ag(\ve^{-1}x)\,dx=|A|\overline{g}
\end{equation}
for any Lebesgue measurable bounded set $A\subset\R^n$, where $|A|$ stands for the Lebesgue measure of $A$.
This is equivalent to say that $g(\ve^{-1}x)$ converges, in the
duality with $L^\infty$ and compactly supported functions, to the constant $\mv{g}$.
Also,
if $A_t:=\{x\in\R^n\,:\, t^{-1}x\in A\}$ for $t>0$ and $|A|\ne0$, \eqref{e1.2} may be written as
\begin{equation}\label{e1.3}
\lim_{t\to\infty}\frac1{t^n|A|}\int_{A_t}g(x)\,dx=\overline{g}.
\end{equation}
\end{definition}
Also, we will use the notation $\medint_{A}g\,dx$ for the average or mean value of $g$ on the measurable set $A$, and $\medint_{\R^n}g\,dx$ or $M(g)$ for $\overline{g}$, given by \eqref{e1.3}.

We recall now the definition of algebras with mean value introduced in \cite{MR704444}. As usual, we denote by $\BUC(\R^n)$ the
space of the bounded uniformly continuous real-valued functions in $\R^n$.

\begin{definition}\label{D:5} Let $\AA$ be a linear subspace of $\BUC(\R^n)$.
We say that $\AA$ is an {\em algebra with mean value} (or {\em
algebra w.m.v.}, in short), if the following conditions are
satisfied:
\begin{enumerate}
\item[(A)] If $f$ and $g$ belong to $\AA$, then the product $fg$ belongs to $\AA$.
\item[(B)] $\AA$ is invariant with respect to translations $\tau_y$ in $\R^n$.
\item[(C)] Any $f\in\AA$ possesses a mean value.
\item[(D)] $\AA$ is closed in $\BUC(\R^n)$ and contains the unity, i.e., the function $e(x):=1$ for $x\in\R^n$.
\end{enumerate}
\end{definition}


For the development of the homogenization theory in algebras
with mean value, as is done in \cite{MR704444,MR1329546} (see also \cite{MR1987520}),
in similarity with the case of almost periodic functions, one
introduces, for $1\leq p<\infty$, the space  $\BB^p$ as the abstract
completion of the algebra
$\AA$ with respect to the Besicovitch seminorm
$$
|f|_p^p:=\limsup_{L\to\infty}\frac{1}{(2L)^n}\int_{[-L,L]^n}|f|^p\,dx.
$$
Both the action of translations and the mean value
extend by continuity to $\BB^p$, and we will keep using the notation
$f(\cdot+y)$ and $M(f)$ even when $f\in\BB^p$ and $y\in\R^n$. Furthermore,
for $p>1$ the product in the algebra extends to a bilinear operator from $\BB^p\times\BB^q$ into $\BB^1$,
with $q$ equal to the dual exponent of $p$, satisfying
$$
|fg|_1\leq |f|_p|g|_q.
$$
In particular, the operator $M(fg)$ provides a nonnegative definite
bilinear form on $\BB^2$.

Since there is an obvious inclusion between this family of spaces,
we may define the space $\BB^\infty$ as follows:
$$
\BB^\infty=\{f\in \bigcap_{1\leq p<\infty}\BB^p\,:\,\sup_{1\le p<\infty}|f|_p<\infty\},
$$
We endow $\BB^\infty$  with the (semi)norm
$$
|f|_\infty:=\sup_{1\le p<\infty}|f|_p.
$$
Obviously the corresponding quotient spaces for all these spaces
(with respect to the null space of the seminorms) are Banach spaces,
and we get a Hilbert space in the case $p=2$. We denote by
$\overset{\BB^p}{=}$, the equivalence relation given by the equality
in the sense of the $\BB^p$ semi-norm.

\begin{remark}\label{R:0.1} A classical argument going back to Besicovitch~\cite{MR0068029} (see also \cite{MR1329546}, p.239) shows that the elements of $\BB^p$ can be represented
by functions in $L_{\loc}^p(\R^n)$, $1\le p<\infty$.
\end{remark}

We next recall a result established in \cite{MR2498566} which provides a connection between algebras with mean value and compactifications  of $\R^n$ endowed with a group of ``translations'' and an invariant probability measure.

\begin{theorem}[cf.\ \cite{MR2498566}]\label{T1}
For an algebra w.m.v.\ $\AA$, we have:
\begin{enumerate} 
\item[(i)] There exist a compact space
${\mathcal K}$ and an isometric isomorphism $i$ identifying $\AA$ with the
algebra $C({\mathcal K})$ of continuous functions on ${\mathcal K}$. By abuse of notation we will make the identification $i(f)\equiv f$, for all $f\in\AA$. 
\medskip

\item[(ii)] The translations $T(y):\R^n\to\R^n$, $T(y)x=x+y$,
extend to a group of homeomorphisms $T(y):{\mathcal K}\to{\mathcal K}$, $y\in\R^n$. The map $T:\R^n \X\KK\to\KK$, given by $T(y,z):=T(y)z$ is continuous. In other words, $T(y)$, $y\in\R^n$, is a ($n$-dimensional) dynamical system over $\KK$. 
\medskip

\item[(iii)] There exists a Radon probability measure ${\mathfrak m}$ on
${\mathcal K}$ which is invariant by the
group of transformations $T(y)$, $y\in\R^n$, such that
$$
\Medint_{\R^n}f\,dx=\int_{\mathcal K} f\,d\mathfrak m.
$$
\medskip
\item[(iv)]For $1\le p\le \infty$, the
Besicovitch space $\BB^p\big/\overset{\BB^p}{=}$ is
isometrically isomorphic to $L^p({\mathcal K}, {\mathfrak m})$.
\end{enumerate}
Actually, {\rm(i)} and {\rm(ii)}  hold independently of the mean value property {\rm(C)} in the definition of algebra w.m.v. 
\end{theorem}

A group of unitary operators
$T(y):\BB^2\to\BB^2$ is then defined by setting $[T(y)f](\cdot) := f(T(y,\cdot))$. Since the elements of $\AA$ are uniformly continuous in $\R^n$,
the group $\{T(y)\}$ is strongly continuous, i.e. $T(y)f\to f$
in $\BB^2$ as $y\to 0$ for all $f\in\BB^2$. The notion of invariant
function is introduced then by simply
saying that a function in $\BB^2$ is {\em invariant} if
$T(y)f\overset{\BB^2}{=} f$, for all $y\in\R^n$. More clearly,
$f\in\BB^2$ is invariant if
\begin{equation}\label{e1.INV}
M\bigl(|T(y)f-f|^2\bigr)=0,\qquad \forall y\in\R^n.
\end{equation}
The concept of ergodic algebra is then introduced as follows.

\begin{definition}\label{D:6} An algebra $\AA$ w.m.v.\ is called {\em ergodic} if any invariant function
$f$ belonging to the corresponding space $\BB^2$ is equivalent (in $\BB^2$) to a constant.
\end{definition}

In \cite{MR1329546}  it is also given an alternative definition of ergodic
algebra which is shown therein to be equivalent to
Definition~\ref{D:6}, by using  von~Neumann mean ergodic theorem.
We state that as the following lemma, whose detailed proof may be
found in \cite{MR1329546}, p.247.

\begin{lemma}\label{L:1.6} Let $\AA$
be an algebra with mean value on $\R^n$. Then $\AA$ is ergodic
if and only if
\begin{equation}\label{eL1.6}
\lim_{t\to\infty}M_y\left(\bigl|\frac{1}{|B(0;t)|}
\int_{B(0;t)}f(x+y)\,dx-M(f)\bigr|^2\right)=0
\qquad\forall f\in\AA.
\end{equation}
\end{lemma}

We next recall some important facts in the theory of ergodic algebras. Let $\CS(\R^n)$ denote the Schwartz space of fast decreasing $C^\infty$ functions. Given $f\in L^\infty(\R^n)$, let us denote by $\hat f$ the distributional Fourier transform of $f$, defined by
$$
\la \hat f,\varphi\ra:=\la f, \hat\varphi\ra,\quad\text{where}\quad \hat \varphi (\xi):=\frac1{(2\pi)^{n/2}}\int_{\R^n} \varphi(x) e^{-\xi\cdot x}\,dx,
$$
for all $\varphi\in\CS(\R^n)$, the last identity being the usual definition of Fourier transform in $\CS(\R^n)$. 

Given an ergodic algebra $\AA$, let us denote 
\begin{equation}\label{eZ(A)}
Z(\AA):= \{f\in\AA\,:\,\text{ $\hat f$ has compact support in $\R^n$ and $0\notin \supp \hat f$} \}.
\end{equation}
The following is a fundamental result in the theory of ergodic algrebras by Zhikov and Krivenko \cite{MR704444}.

\begin{lemma}\label{L:ZK} For any ergodic algebra $\AA$, $Z(\AA)$ is dense in $V=\{ f\in\BB^2\,:\, M(f)=0\}$, with respect to   the  $\BB^2$-norm.
\end{lemma}

The following result from  \cite{MR2498566} is an easy consequence of Lemma~\ref{L:ZK}, and will be used later, in this paper, in our  analysis of the homogenization problem.
\begin{lemma}\label{Lapla}
Let $\AA$ be an ergodic algebra on $\R^n$ and $h\in \BB^2$ such that
$\overline{h\,\Delta f}=0$, for all $f\in \AA$ such that $\Delta f\in \AA$. Then,
$h$ is $\BB^2$-equivalent to a constant.
\end{lemma}

 For an ergodic algebra $\AA$ with $\BB^2$ the associated generalized Besicovitch space, we now define the spaces ${\BB}^2_{{\pot}}$,  ${\BB}^2_{{\sol}}$ ({\em cf.} \cite{MR1329546}).

 We say that a vector field
$v=(v^1,\cdots,v^n) \in{\BB}^2(\R^n;\R^n) $ is a {\em potential field} if
\begin{equation}\label{e2.1}
\int_{\KK}\bigg(v^i(z) {\partial}_{j} \varphi(z)- v^j(z) {\partial}_{i}\varphi(z)\bigg)\,d\mm(z) =0,\quad \text{for all $i,j\in \{1,\cdots,n\}$}, 
\end{equation}
for all $\varphi\in\AA$ with all its derivatives in $\AA$, where, as usual, $\partial_i$ denotes the partial derivative with respect to the $i$-th coordinate $x_i$.  We then  define
 ${\BB}^2_{{\pot}}$ as the subset of potential vector fields of ${\BB}^2(\R^n;\R^n) $. 
  
 Similarly,
we say that $v \in {\BB}^2(\R^n;\R^n)$ is a {\em solenoidal field} if
\begin{equation}\label{e2.2}
\int_{\KK}v(z)\cdot \nabla \varphi(z)\,d\mm(z)=0,
\end{equation}
for all $\varphi$ as above, and  define ${\BB}^2_{{\sol}}$ as the subset of solenoidal vector fields of ${\BB}^2(\R^n;\R^n) $. 

We also introduce the spaces ${\mathcal V}^2_{{\pot}}=\{v \in {\BB}^2_{{\pot}}; ~ \overline{v}=0 \}$,  and ${\mathcal V}^2_{{\sol}}=\{v \in {\BB}^2_{{\sol}}; ~ \overline{v}=0 \}$. 

The following  fundamental result is one of the main motivations for Lemma~\ref{L:ZK}  and easily follows  from the latter ({\em cf.}  \cite{MR1329546}).
\begin{lemma}[{\em cf.} \cite{MR1329546}]\label{L:ZK0}
The following orthogonal decomposition  holds:
\begin{equation}\label{e2.3}
\BB^2(\R^n;\R^n)={\mathcal V}^2_{{\pot}}\oplus {\mathcal V}^2_{{\sol}}\oplus\R^n.
\end{equation}
\end{lemma}

In particular, since the spaces  ${\mathcal V}^2_{{\pot}}$ and ${\mathcal V}^2_{{\sol}}$ are orthogonal we have
\begin{equation}\label{e2.4}
\overline{v \cdot w} = \overline{v} \cdot \overline{w}, \quad\text{ for all  $v \in {\BB}^2_{{\pot}}$ and  $w \in {\BB}^2_{{\sol}}$}.
\end{equation}
Here, for a vector-valued function $f=(f_1,\cdots,f_n)$, as natural,  we denote  $\mv f:=(\mv f_1,\cdots, \mv f_n)$.

\subsection{Weakly Almost Periodic Functions} \label{S:2.1}

Examples of ergodic algebras include the periodic continuous functions, the almost periodic functions, and the Fourier-Stieltjes transforms, studied in \cite{MR2556576}. 
More generally, all the just mentioned ergodic algebras are subalgebras of a strictly larger ergodic algebra,  that is the algebra of the (real-valued) {\em weakly almost periodic functions} in $\R^n$, $\operatorname{WAP}(\R^n)$. It is defined as the subspace of the space of the bounded continuous functions,  $C_b(\R^n)$, formed by those  $f: \R^n\to\R$, satisfying the property that any sequence of its translates $(f(\cdot+\l_i))_{i\in\N}$ possesses a subsequence $(f(\cdot+\l_{i_k}))_{k\in\N}$ weakly converging in $C_b(\R^n)$, where the latter denotes the space of the bounded continuous functions in $\R^n$.
This space was introduced and its main properties were obtained by Eberlein in \cite{MR0036455} (see also \cite{MR0082627}). In particular, in \cite{MR0036455}, Eberlein proved that $\operatorname{WAP}(\R^n)$ satisfies all the properties defining an algebra w.m.v. It is immediate to see, from the definition, that $\operatorname{WAP}(\R^n)\supset \AP(\R^n)$,  where the latter denotes the space of almost periodic functions. Indeed, for functions in $\AP(\R^n)$, Bochner theorem gives the relative compactness of the translates $f(\cdot+\l)$, $\l\in\R^n$,  in the $\sup$-norm (see, e.g., \cite{MR0068029}). We summarize in the following lemma the properties of $\WAP(\R^n)$ which were essentially proved by Eberlein in   \cite{MR0036455}.

\begin{lemma}[{\em cf.}  \cite{MR0036455}] \label{L:2.wap} $\WAP(\R^n)$ is an ergodic algebra which contains the algebra of Fourier-Stieltjes transforms 
$\operatorname{FS}(\R^n)$.
\end{lemma}

\begin{proof} The fact that  $\operatorname{WAP}(\R^n)\subset \BUC(\R^n)$ is proved by contradiction. Assume, on the contrary, that one can find points $\xi_k,\s_k$, with $|\xi_k-\s_k|\to0$ as $k\to\infty$, such that $|f(\xi_k)-f(\s_k)|\ge\ve_0>0$, for all $k\in\N$. Define $g_k(x)=f(x+\xi_k)-f(x+\s_k)$. By passing to a subsequence, we may assume that  $g_k$ converges weakly to some $g\in C_b(\R^n)$; in particular $|g(0)|\ge\ve_0>0$. On the other hand, if $B_r(\xi)$ is the ball of radius $r>0$ around $\xi\in\R^n$, 
  \begin{multline*}
 \left|\int_{B_r(0)}g_k(x)\,dx\right|\le\left|\int_{B_r(\xi_k)} f(x)\,dx-\int_{B_r(\s_k)}f(x)\,dx\right|\le \|f\|_\infty \left|\left(B_r(\xi_k)\setminus B_r(\s_k)\right)\cup \left(B_r(\s_k)\setminus B_r(\xi_k)\right) \right|
  \\= \|f\|_\infty \bigl|\left(B_r(0)\setminus B_r(\xi_k-\s_k)\right)\cup \left(B_r(\xi_k-\s_k)\setminus B_r(0)\right)\bigr|\to0, \quad\text{as $k\to\infty$, for all $r>0$},
  \end{multline*}
  which gives the desired contradiction.  We also remark that, if $g\in C_b(\R^n)$ is the weak limit of a sequence of translates $f(\cdot+\l_k)$, with $f\in \WAP(\R^n)$, then 
  $g\in\BUC(\R^n)$. Indeed, weak convergence impies pointwise convergence in $\R^n$, in particular, and so, since the family $\{f(\cdot+\l_k)\}$ is equicontinuous, for 
  $f\in\BUC(\R^n)$,  it follows that $g\in\BUC(\R^n)$. 

To have a better idea of this space,  consider \v Cech compactification of $\R^n$, associated with the algebra $C_b(\R^n)$ (see, e.g., \cite{MR1009162}), denote it by
 $\KK_0$. There is an isometric isomorphism between $C_b(\R^n)$ and $C(\KK_0)$, and weak convergence in $C_b(\R^n)$ is then translated to pointwise convergence in 
 $C(\KK_0)$. So, the weakly almost periodic functions are then identified with the functions in $C(\KK_0)$ whose sequences of translates, $(f(\cdot+\l_i))_{i\in\N}$, always possess a subsequence converging pointwise to a function $g\in C(\KK_0)$. By this characterization, it is immediate that $\operatorname{WAP}(\R^n)$ is an algebra in $C_b(\R^n)$, closed in the $\sup$ norm. For the following considerations on $\WAP(\R^n)$, instead of the compactification provided by all space $C_b(\R^n)$, it will be more convenient to consider the compactification provided by the algebra $\BUC(\R^n)$, which is then identified with the compact $\KK_0/\sim$ with the topology $\tau_0$ generated by the functions in 
 $\BUC(\R^n)$, where $\sim$ is the equivalence relation whose quotient makes $\tau_0$ Hausdorff. So, we have the identification of $\BUC(\R^n)$ with the space of continuous functions $C(\KK_0/\sim,\tau_0)$. In what follows we omit the quotient, writing simply $\KK_0$, instead of $\KK_0/\sim$, and will assume $\KK_0$ to be endowed with the topology 
 $\tau_0$.

  Existence of mean value for functions in $\operatorname{WAP}(\R^n)$ may be seen as follows.  First, by Theorem~\ref{T1}, the translations $T(y)f(\cdot)=f(\cdot+y)$ may be extended to $\KK_0$ to form a continuous dynamical system in $\KK_0$. A well known theorem by Krylov and Bogolyubov  asserts the existence of a probability measure $\mu$ in $\KK_0$, invariant by $\{T(y)\,:\,y\in\R^n\}$ (see, e.g., \cite{MR0121520}; the extension of the proof given therein, for compact metric spaces, to general compact topological spaces is straightforward). Also,  von~Neumann mean ergodic theorem (see, e.g., \cite{MR1009162}) implies that, given $f\in \operatorname{WAP}(\R^n)$, $M_L(f)(z):=\medint_{B_L(0)}f(T(y)z)\,dy$ converges, as $L\to\infty$, in $L^2(\KK_0,\mu)$, to a function $g(z)\in L^2(\KK_0,\mu)$ which is invariant, that is, $g(z+y)=g(z)$, for $\mu$-a.e.\ $z\in\KK_0$,  for all $y\in\R^n$. Observe that, for any $\xi\in\R^n$, 
 $$
 \Medint_{B_L(\xi)}f(T(y)\,\cdot )\,dy=\Medint_{B_L(0)}f(T(y+\xi)\,\cdot)\,dy=T(\xi)M_L(f)(\cdot)\to T(\xi)g(\cdot)=g(\cdot),\quad\text{as $L\to\infty$, in $L^2(\KK_0,\mu)$},
 $$ 
 by the continuity of $T(\xi):L^2(\KK_0,\mu)\to L^2(\KK_0,\mu)$, and the invariance of $g$. 
 Now, $M_L(f)(z)$ may be arbitrarily approximated in $C(\KK_0)$ by a finite convex combination of translates of $f$, $g_L(\cdot)=\theta_L^1 f(\cdot+\l_L^1)+\cdots+\theta_L^{K(L)}f(\cdot+\l_L^{K(L)})$, and,  taking $L=1,2,\cdots$, we may arrange that $g_L\to g$, in $L^2(\KK_0,\mu)$. Let us consider the separable closed subspace  $S\subset C(\KK_0)$ generated by the translates of $f$, $f(\cdot+\l)$, $\l\in\R^n$. The dual of $S$,  is a separable space which, by Hahn-Banach, may be viewed as a subspace of the dual of $C(\KK_0)$.
We may then define a metric $d(f,g)$ in $S$, whose induced topology is equivalent to the weak topology of $S$, and satisfies $d(f+h,g+h)=d(f,g)$.   Since the set $O(f)=\{f(\cdot+\l)\,:\,\l\in\R^n\}$ is 
pre-compact, we deduce that it is totally bounded in the metric $d$. But then, since $S$ with the weak topology is locally convex, by a well known result (see, e.g.,  \cite{MR1157815}, p.72) the convex hull of $O(f)$, $co(O(f))$, is totally bounded, and, hence, $\mv{co(O(f))}$ is compact in the weak topology. In particular, by passing to a subsequence if necessary, 
we deduce that $g_L$ weakly converges to some $\tilde g\in C(\KK_0)$, that is $g_L(z)\to \tilde g(z)$, for all $z\in\KK_0$. But then, $\tilde g(z)=g(z)$, $\mu$-a.e., and by the invariance of $g$, we deduce that $g$ is constant  and we denote it by $\mv f$. Hence, for any $\xi\in\R^n$,  and all $z\in \KK_0$,  the averages 
  $\medint_{B_L(\xi)}f(T(y)\,z )\,dy$ converge to $\mv f$, which does not depend on either $z$ or $\xi$, and this implies that $f$ possesses mean value and this  is $\mv f$. 
  
 Taking the invariant measure $\mu$, above, as the measure induced by the mean value, we see that  the proof just given for the existence of the mean value for functions in 
 $\operatorname{WAP}(\R^n)$ may be repeated, line by line, to prove  the ergodicity of this algebra w.m.v., as a straightforward application of Lemma~\ref{L:1.6}. In sum, 
 $\operatorname{WAP}(\R^n)$ is an ergodic algebra.

 We recall that  the Fourier-Stieltjes algebra $\operatorname{FS}(\R^n)$ is defined as the closure in the $\sup$-norm of functions $f:\R^n\to\R$ which admit a representation as 
 \begin{equation}\label{ewap1}
 f(x)=\int_{\R^n}e^{ix\cdot y}\,d\mu(y),
 \end{equation}
 for some signed Radon measure in $\R^n$ with finite total variation. If $f$ admits the representation in \eqref{ewap1}, then any of its translates, $f(\cdot+\l)$, admits a similar
 representation with $\mu(y)$ replaced by $e^{i\l\cdot y}\mu(y)$. Suppose first, that $f\in\operatorname{FS}(\R^n)$, admits a representation as in \eqref{ewap1}, with $\supp\mu\subset  B_R(0)$,  for some $R>0$.   Given a sequence of translates, $f(\cdot+\l_n)$, we have that these translates satisfy an equation like
 \eqref{ewap1}, with $\mu(y)$ replaced by $\mu_n(y):=e^{\l_n\cdot y}\mu(y)$, and so $|\mu_n|(\R^n)=|\mu|(\R^n)$, and $\supp\mu_n=\supp\mu$. Since the space of Radon measures with finite total variation and support in a compact $K\subset\R^n$, $\M(K)$, is the dual of $C(K)$, we may extract a subsequence from $\mu_n$, still labeled $\mu_n$,
 such that $\mu_n\wto\nu$ in the weak-star topology of $\M(K)$, for some $\nu\in\M(K)$. Therefore, $f(\cdot+\l_n)$ pointwise converges to $g\in C_b(\R^n)$, where
 $$
 g(x)=\int_{\R^n}e^{ix\cdot y}\nu(y).
 $$
 Since any function in $\operatorname{FS}(\R^n)$ is the uniform limit of functions satisfying a representation like \eqref{ewap1}, for a signed Radon measure $\mu$ with compact support and finite total variation, we conclude that  $\operatorname{FS}(\R^n)\subset  \operatorname{WAP}(\R^n)$. 
\end{proof}

 Finally, Rudin, in \cite{MR0102705}, proved that there are functions in  $\operatorname{WAP}(\R^n)$ which are not in  $\operatorname{FS}(\R^n)$. This shows that the inclusion  $\operatorname{FS}(\R^n)\subset \operatorname{WAP}(\R^n)$ is strict and answers positively the question, posed in \cite{MR2556576},   whether there are ergodic algebras containing strictly $\operatorname{FS}(\R^n)$, as observed in \cite{MR2733214}.

\section{Two-scale Young Measures}\label{S:3}

In this section we recall the theorem giving the existence  of two-scale Young measures established in \cite{MR2498566}. We will state it in a more general fashion, suitable for dealing with sequences of functions uniformly bounded in $L^p$, for any $1<p\le\infty$,  and we will also outline  the extension of the proof of the  corresponding result in 
\cite{MR2498566} to the more general formulation given here. We also include in this section a discussion about two-scale convergence and its relation with two-scale Young measures, in the general context of algebras w.m.v.  
We begin by recalling the concept of vector-valued
algebra with mean value.

Given a Banach space $E$ and an algebra w.m.v.\ $\AA$,
we denote by $\AA(\re^n;E)$ the space
of functions $f\in\BUC(\R^n;E)$ such that $L_f:=\la L,f\ra$ belongs to
$\AA$ for all $L\in E^*$ and the family
$\{L_f\,:\, L\in E^*,\ \|L\|\le 1 \}$ is relatively
compact in $\AA$.

As an example, let us consider a function $g(z,x)$ in the space  $\AA(\R^n; L^\infty(\Om))$, where $\Om\subset\R^n$ is a bounded open set. 
 For each fixed $x\in\Om$, and $\d>0$, sufficiently small,  the average 
\begin{equation}\label{e3.0}
\la L_{x,\d}, g(z,\cdot)\ra :=\Medint_{B_\d(x)}g(z,y)\,dy,
\end{equation}
is an element of $L^\infty(\Om)^*$, with $\|L_{x,\d}\|\le 1$.  As a consequence of the  just given definition, we have  that the family $g_{x,\d}(z):=\la L_{x,\d}, g(z, \cdot)\ra$, 
$\d>0$, is relatively compact in $\AA$. Suppose $x$ is a Lebesgue point of $g(z,\cdot)$, for all $z\in\Q^n$; the set $\tilde\Om$ of such $x\in\Om$ has total Lebesgue measure.  
Given any sequence $\d_i\to0$, we may extract a subsequence, still denoted $\d_i$, such that $g_{x,\d_i}(z)$ converges uniformly in $\R^n$. Since, for each   $z\in\Q^n$, this limit exists and coincides with $g(z,x)$, independently of the subsequence, we deduce that the limit of the whole sequence $g_{x,\d}(z)$ exists for all $z\in\R^n$. 
Therefore, we may take a representative of $g(z,x)$ so that for $x\in\tilde\Om$ and all $z\in\R^n$ we have 
$$
g(z,x)=\lim_{\d\to0}g_{x,\d}(z),
$$
and the limit is uniform with respect to $z\in\R^n$. In particular, by dominated convergence, we have that 
$$
\lim_{\d\to0}\int_{\Om} \sup_{z\in\R^n}|g_{x,\d}(z)-g(x,z)|^p\,dx=0,
$$  
 holds for $g\in\AA(\R^n;L^\infty(\Om))$ and any $p\ge1$.

 For bounded Borel sets
$Q\subset\R^n$ and $f\in\BUC(\R^n;E)$, it is easily checked by an
approximation with Riemann sums that $L\mapsto \int_Q\langle
L,f\rangle\,dx$ defines a linear functional on $E^*$, continuous for
the weak topology $\sigma(E^*,E)$; as a consequence, there exists a
unique element of $E$, that we shall denote by $\int_Q f\,dx$,
satisfying
$$
\langle L,\int_Q f\,dx\rangle=\int_Q\langle
L,f\rangle\,dx\qquad\forall L\in E^*.
$$
For similar reasons, if $f\in\AA(\R^n;E)$ the integrals
$\medint_{Q_t} f\,dx$ weakly converge in $E$, as $t\to +\infty$, to a
vector, that we shall denote by $\medint_{\R^n}f\,dx$, characterized
by
$$
\la L,\Medint_{\R^n}f\,dx\ra=\Medint_{\R^n}\la L,f\ra\,dx\qquad\forall L\in E^*.
$$

\begin{theorem}[cf.\ \cite{MR2498566}]\label{T2}
Let $E$ be a Banach space, $\AA$ an algebra and ${\mathcal K}$ be the compact associated with
$\AA$. There is an isometric isomorphism between $\AA(\R^n;E)$ and
$C({\mathcal K};E)$. Denoting by $g\mapsto\util{g}$ the canonical map from
$\AA$ to $C({\mathcal K})$, the isomorphism associates to $f\in\AA(\R^n;E)$
the map $\tilde{f}\in C({\mathcal K};E)$ satisfying
\begin{equation}\label{mio2}
\util{\langle L,f\rangle}=\langle L,\tilde{f}\rangle\in C({\mathcal K})
\qquad\forall L\in E^*.
\end{equation}
In particular, for each $f\in \AA(\R^n;E)$, $\|f\|_E\in\AA$. As before we will make the identification $\util{g}\equiv g$. 
\end{theorem}

We
define the space $L^p({\mathcal K};E)$ as the completion of $C({\mathcal K};E)$
with respect to the norm $\|\,\cdot\,\|_p$, defined as usual:
$$
\|f\|_p:=\left(\int_{{\mathcal K}}\|f\|_E^p\,d{\mathfrak m}\right)^{1/p}.
$$
As usual, we identify functions in $L^p$ that coincide ${\mathfrak m}$-a.e.
in ${\mathcal K}$.

The next theorem gives the existence of two-scale Young measures associated
with an algebra $\AA$, and it was established in \cite{MR2498566} (see also \cite{MR2481061}). Actually, in the following, we state a more general version of the corresponding theorem in  \cite{MR2498566}, which is suitable for sequences that are uniformly bounded in $L^p$, for any  $p>1$,  as well as $p=\infty$, which is based on remark~3.1, of 
\cite{MR2481061}.   For simplicity we only consider the case where the functions in the sequence take values in $\R^m$, for some $m\in\N$. Before stating the theorem, we introduce some notations.

In the next statement, we make use of the terminology of nets and subnets, indexed by directed sets,   which are generalizations of the usual concepts of sequences and subsequences indexed by natural numbers and increasing sequences of natural numbers. These concepts become necessary since, in general, the algebra $\AA$ is not a  separable space, as is the case of the space of almost periodic functions, for instance. We briefly recall here the definitions for these concepts taken exactly as they are formulated in   \cite{MR0070144}, and refer
to the latter for a more detailed discussion. 

A binary relation $\ge$ directs a set $D$ if $D$ is non-void and:
\begin{enumerate}

\item[(n1)] if $m,n$ and $p$ are members of $D$ such that $m\ge n$ and $n\ge p$, then $m\ge p$;

\item[(n2)]  if $m\in D$, then $m \ge m$;   

\item[(n3)] if $m$ and $n$ are members of $D$, then there is $p\in D$ such that $p\ge m$ and $p\ge n$.

\end{enumerate}
A {\em directed set} is a pair $(D,\ge)$ such that $\ge$ directs $D$. 

A {\em net} of elements of a set $A$ is a pair $(S,\ge)$ such that $S: D\to A$ is a function and $\ge$ directs the domain $D$ of $S$. A net $\{T_m\,:\, m\in D\}$ is a subnet of a net $\{S_n\,:\, n\in E\}$ if there is a function $N:D\to E$ such that:
\begin{enumerate}

\item[(s1)] $T=S\circ N$, that is, $T_i=S_{N_i}$, for each $i\in D$;

\item[(s2)] for each $m\in E$ there is $n\in D$ with the property that, if $p\ge n$, then $N_p\ge m$.

\end{enumerate}
If there is $m\in D$ such that $T$ restricted to the directed set $D_m=\{ p\in D\,:\, p\ge m\}$ is a subnet of $S$, we will say that $T$ is {\em eventually} a subnet of $S$.  

Given a  succession of nets $\{S_i\}_{i\in\N}$, such that $S_{i+1}$ is a subnet of $S_i$, so that there is a function $N^{(i)}: D_{i+1}\to D_{i}$,  where $D_i$ is the domain of 
$S_i$,  $i\in\N$, as in the definition of subnet just recalled, we may form the {\em diagonal subnet} $S$, which is  eventually  a subnet of any of the nets $S_i$, $i\in\N$, 
  as follows. We  define the directed set $D:=\{ (i,d)\,:\, i\in\N,\ d\in D_i\}$, directed by $\ge$, given by $(i,d)\ge (j,e)$ if,  either $i= j$, and $d\ge_i e$, where $\ge_i$ directs $D_i$, or
  $i>j$, and $N^{(j-1)}\circ \cdots\circ N^{(i)}(d)\ge_j e$. Clearly, $S$, with domain $D$ directed by $\ge$ so defined,  is eventually a subnet of any $S_i$.

For a set $A\subset\R^n$, we denote by $|A|$ the Lebesgue measure of $A$. If $(F,\mathcal{F})$ is a measurable space, we say that a parametrized family of probability measures in $\R^m$, $\{\nu_\xi\}_{\xi\in F}$, is weakly measurable, if $\la \nu_\xi, \varphi\ra$ is ($\mathcal{F}$-)measurable for all $\varphi\in C_0(\R^m)$, where, given any locally compact topological space, we denote by $C_0(N)$ the space of continuous functions with compact support. We also denote by $C_b(N)$ the space of continuous functions that are bounded in $N$.  Both $C_0(N)$ and $C_b(N)$ are endowed with the sup norm, as usual. 

 Let $\AA$ be an algebra w.m.v.\  to which all spaces $\BB^p$ in the next statement refer as well as their vector-valued extensions. Let also $\KK$ be the compact topological  space and   $\mm$ be the invariant measure on $\KK$ associated with $\AA$ by Theorem~\ref{T1}.

\begin{theorem}\label{T:3.1}
Let $\Om\subset
\re^n$ be a bounded open set and $\{u_\ve(x)\}_{\ve>0}$ be a family
of Lebesgue measurable $\R^m$-valued functions, satisfying
\begin{equation}\label{eT3.1}
\lim_{R\to\infty}\limsup_{\ve\to0}|\{|u_\ve|>R\}|=0.
\end{equation}
Given any sequence $\ve_i\downarrow 0$, $i\in\N$,  there exist a
subnet $\{u_{\ve_{i(d)}}\}_{d\in D}$, indexed by a certain directed
set $D$, and a family of probability measures on $\R^m$,
$\{\nu_{z,x}\}_{z\in {\mathcal K}, x\in \Om}$, weakly measurable with respect
to the product of the Borel $\sigma$-algebras in ${\mathcal K}$ and $\R^n$,
such that
\begin{equation}\label{young}
\lim_{D}\int_{\Om}\Phi(\frac{x}{\ve_{i(d)}},x,u_{\ve_{i(d)}}(x))\,dx=
\int_{\Om}\int_{{\mathcal K}}\la\nu_{z,x},\Phi(\cdot,z,x)\ra\,d{\mathfrak m}(z)\,dx,
\qquad \text{for all $\Phi\in\AA\left(\R^n;C_0(\Om; C_b(\R^m)\right)$}.
\end{equation}
Moreover, equality \eqref{young} still holds for functions $\Phi$ in the spaces:
\begin{enumerate}
\item[(1)]
$\BB^1\bigl(\R^n;C_0(\Om; C_b(\R^m)\bigr)$;
\medskip
\item[(2)]
$\BB^p\bigl(\R^n;C(\bar \Om; C_b(\R^m))\bigr)$, with $p>1$;
\medskip
\item[(3)]
$L^1\bigl(\Om;\AA(\R^n;C_b(\R^m))\bigr)$.
\end{enumerate}
Furthermore, if $\{u_\ve(x)\}_{\ve>0}$ is uniformly bounded in $L^r(\Om;\R^m)$, for some $r>1$, in which case \eqref{eT3.1} is trivially verified, then the relation in \eqref{young} 
also holds if $\Phi(z,x,\l)/(1+|\l|)^s$ belongs to any of the spaces, $\AA(\R^n; C_0(\Om; C_b(\R^m)))$,  {\rm (1), (2), or (3)} above, for some $1<s<r$.
\end{theorem}

\begin{proof} The proof of the first assertion in the statement is obtained by first applying theorem~6.1  of \cite{MR2498566} (see  \cite{MR2481061}, for the proof) to the truncated sequences $u_\ve^R(x):=\rho_R(u_\ve(x))$, with  $\rho_R:\R^m\to\R^m$ given by $\rho_R(\l)=\l$, if $|\l|<R$, $\rho_R(\l)=0$, if $|\l|\ge R$, $R\in \N$. This application leads to two-scale Young measures $\nu_{z,x}^R$, generated by subnets $u_{\ve(d_R)}^R$, where, for each $R\in\N$,  $\ve(d_{R+1})$ is a subnet of $\ve(d_R)$. 
The measures $\mu^R:=d\mm(z)\,dx \otimes \nu_{z,x}^R$ weakly converge to a measure $\mu$. Indeed, we observe first that 
$\la\mu^{R+k},\varphi\ra=\la\mu^R,\varphi\ra$, for all $\varphi\in C_0(\KK\X\Om\X B_R(0))$, for all $k\in\N$. Therefore, if we define the linear functional $\mu$ over  
$C_0(\KK\X\Om\X\R^m)$ by
\begin{equation}\label{eT3.1_1}
\la\mu, \varphi\ra:=\lim_{R\to\infty}\la \mu^R,\varphi\ra,
\end{equation}
we have that $\mu$ is well defined since, if $\varphi\in C_0(\KK\X\Om\X\R^m)$, then $\varphi\in C_0(\KK\X\Om\X B_{R_0}(0))$, for some $R_0\in\N$, depending on $\varphi$, so that, the above limit exists and is simply $\la \mu^{R_0},\varphi\ra$. Also, since the $\mu^R$ are positive measures in $\KK\X\Om\X\R^m$ with uniformly bounded total variation,
then  $\mu$ is a bounded linear functional over $C_0(\KK\X\Om\X\R^m)$, which is nonnegative, and so $\mu$ is a positive Radon measure over $\KK\X\Om\X\R^m$, by Riesz representation theorem. 

We now define the subnet for which \eqref{young} holds, as we will see,  by taking the diagonal subnet from $u_{\ve(d_R)}$, $R\in\N$. We define the two-scale family of measures $\nu_{z,x}$  by the disintegration $\mu=d\mm(z)\,dx\otimes\nu_{z,x}$, which follows from the fact that the projection of $\mu$ on $\KK\X\Om$ must coincide with
$d\mm(z)\,dx$, since this is true for all $\mu^R$, $R\in\N$. Now, given $\Phi\in \AA\left(\R^n; C_0(\Om\X C_b(\R^m))\right)$, we claim that \eqref{young} holds, for the diagonal subnet $\ve(d)$ defined above. Indeed, $\Phi(z,x,\l)$ is uniformly bounded, so there is $M>0$ such that $|\Phi(z,x,\l)|\le M$, for all $(z,x,\l)\in\R^n\X\Om\X\R^m$. We may assume, for simplicity, that $\Phi(z,x,0)=0$, for all $(z,x)\in\R^n\X\Om$, since, otherwise, we only need to prove \eqref{young} for $\Phi(z,x,\l)-\Phi(z,x,0)$.
Given $\g>0$,
by using \eqref{eT3.1} we can obtain $R>0$ such that 
\begin{equation}\label{eT3.1_2}
\int_{\Om}|\Phi(\frac{x}{\ve},x,u_\ve(x))|(1- \xi_R(u_\ve(x)))\,dx<\frac\g3, \quad\text{for all $\ve>0$},
\end{equation}
where $\xi_R(x)=\xi(Rx)$, and $\xi\in C_0^\infty(\R^m)$,  is such that $0\le \xi(\l)\le 1$,  $\xi(\l)=1$, for $|\l|\le 1$, and $\xi(\l)=0$, for $|\l|\ge 2$.  Also, if $R>0$ is sufficiently large
\begin{equation}\label{eT3.1_3}
\left|\int_{\KK\X\Om}\la\nu_{z,x},\Phi(z,x,\cdot)(1-\xi_R(\cdot))\ra\,d\mm(z)\,dx\right|=\left|\la\mu,\Phi(\cdot,\cdot,\cdot)(1-\xi_R(\cdot))\ra\right|<\frac\g3,
\end{equation}
by dominated convergence. On the other hand, we have
\begin{equation}\label{eT3.1_4}
\lim_{D} \int_{\Om}\Phi(\frac{x}{\ve},x,u_\ve(x))\xi_R(u_\ve(x))\,dx=\int_{\KK\X\Om}\la\nu_{z,x},\Phi(z,x,\cdot)\xi_R(\cdot)\ra\,d\mm(z)\,dx,
\end{equation}
since  $\Phi(x/\ve,x,u_\ve(x))\xi_R(u_\ve(x))=\Phi(x/\ve,x,u_\ve^{2R}(x)))\xi_R(u_\ve^{2R}(x))$, where, $ u_\ve^{2R}(x)=\rho_{2R}(u_\ve(x))$, as above, since the representation
holds for $u_\ve^{2R}$, and so, we may find $d_0\in D$ such that, for $d\in D$ with $d\ge d_0$ we have
\begin{equation}\label{eT3.1_5}
\left| \int_{\Om}\Phi(\frac{x}{\ve(d)},x,u_\ve(d)(x))\xi_R(u_\ve(d)(x))\,dx-\int_{\KK\X\Om}\la\nu_{z,x},\Phi(z,x,\cdot)\xi_R(\cdot)\ra\,d\mm(z)\,dx\right|<\frac\g3,
\end{equation}
which proves \eqref{young}. 

Next, assertions~(1), (2) and (3) follow by approximating the functions in the classes described in (1), (2) and (3) by functions in $\AA(\R^n; C_0(\Om; C_b(\R^m))$,
in the same way as done in the proof of the corresponding assertions in theorem~6.1of \cite{MR2498566} (see also  \cite{MR2481061}). 

Now, let us check the case where $u_\ve$ is a bounded sequence in $L^r(\Om)$, for some $r>1$.  Assume, first, that $\Phi(z,x,\l)/(1+|\l|)^s\in \AA(\R^n;C_0(\Om; C_b(\R^m)))$,
for some $1<s<r$. Let $\g>0$ be given. We then verify that inequalities analogous to \eqref{eT3.1_2} and \eqref{eT3.1_3} hold.  Indeed, since $\Phi(z,x,\l)/(1+|\l|)^s\le M$, for all
$(z,x,\l) \in \R^n\X\Om\X\R^m$, for some $M>0$, we have
\begin{equation}\label{eT3.1_2'}
\int_{\Om}|\Phi(\frac{x}{\ve},x,u_\ve(x))|(1- \xi_R(u_\ve(x)))\,dx< \frac1{(1+R)^{r-s}}M\|(1+|u_\ve|)\|_{L^r(\Om)}^r<\frac\g3, \quad\text{for all $\ve>0$},
\end{equation}
is $R>0$ is sufficiently large.  On the other hand, by the monotone sequence theorem, we have
\begin{multline}\label{eT3.1_3'}
\la\mu,(1+|\l|)^s\ra=\lim_{i\to\infty}\la\mu,\min\{(1+|\l|)^s,i\}\ra=\lim_{i\to\infty}\lim_{D}\int_{\Om}\min\{(1+|u_{\ve(d)}(x)|)^s,i\}\,dx\\
\le \limsup_{\ve\to0}\int_{\Om}(1+|u_\ve|)^s)\,dx<+\infty.
\end{multline}
Therefore, since $|\Phi(z,x,\l)|\le M(1+|\l|)^s$, we can obtain the analogue of \eqref{eT3.1_3} again by dominated convergence, using \eqref{eT3.1_3'}. So, we conclude that 
\eqref{young} holds also in this case, reasoning exactly as we did for function $\Phi\in\AA(\R^n;C_0(\Om;C_b(\R^m)))$. The analogues of (1), (2) and (3) are then obtained by approximation by functions $\Phi$ such that $\Phi/(1+|\l|)^s\in \AA(\R^n; C_0(\Om;C_b(\R^m)))$, exactly as in \cite{MR2481061}. 

\end{proof}

It is worthwhile to realize the relationship between the concept of two-scale Young measures and that of two-scale convergence defined as follows.

\begin{definition}\label{D:3.1}  
Let  $\{v_{\ve_i}\}_{i\in\N}$, $\ve_i\downarrow 0$ as $i\to\infty$, be a uniformly bounded sequence in $L^1(\Om)$. We say that $V(z,x)\in \BB^1(\R^n; L^1(\Om))$ is the two-scale limit of $v_{\ve_i}$ (with respect to the algebra w.m.v.\ $\AA$),  or that $v_{\ve_i}(x)$ two-scale converges to $V(x,z)$, if, for any $\varphi\in \AA$ and $\phi\in C_0(\Om)$ , we have 
\begin{equation}\label{eD3.1}
\lim_{i\to\infty} \int_{\Om} v_{\ve_i}(x)\varphi(\frac{x}{\ve_i})\phi(x)\,dx=\int_{\Om}\int_{\KK} V(z,x)\varphi(z)\phi(x)\,d\mm(z)\,dx.
\end{equation}
\end{definition}

Before stating the next theorem, related with two-scale convergence, let us introduce the concept of regular test function  that will appear in the  statement of the next result,
 and will also be important subsequently in this paper.  

\begin{definition}\label{D:3.2} We say that $\psi(z,x)\in \BB^2(\R^n;L^2(\Om))\bigcap L_\loc^2(\R^n;L^2(\Om))$ is a {\em regular test function} if,  for all $\ve>0$,  $x\mapsto \psi(\frac{x}{\ve},x)\in L^2(\Om)$, and satisfies
\begin{equation}\label{eT3.2-2}
\lim_{\d\to0}\, \limsup_{\ve>0} \int_{\Om}\left|\Medint_{B_\d(x)}\psi(\frac{x}{\ve},y)\,dy -\psi(\frac{x}{\ve},x)\right|^2\,dx=0,
\end{equation}
where we set $\psi(z,y)=0$, for all $z\in\R^n$ and $y\in\R^n\setminus\Om$. 
\end{definition}

\begin{remark}\label{R:3.1} Observe that any finite linear combination of functions either of the form $\s(z,x)\zeta(z)$, with $\s\in\AA(\R^n;L^\infty(\Om))$ and 
$\z\in\BB^p(\R^n)\cap L_\loc^p(\R^n)$, with $p>2$, or of the form $\s(x)\zeta(z)$, with $\s\in L^p(\Om)$ and 
$\z\in\BB^q(\R^n)\cap L_\loc^q(\R^n)$,  with $p>2$, $q>2$ and $p^{-1}+q^{-1}=1/2$, or else, of the form $\s(x)\z(z)$, with $\s\in L^2(\Om)$ and $\z\in\AA(\R^n)$,  is a regular test function in the sense of Definition~\ref{D:3.2}. Also, it is immediate to see that any function in $L^2(\Om;\AA)$ is a regular test function.
\end{remark}

It is not so pleasant the fact that, because of the non-separability of $\AA$, in general, the statement of Theorem~\ref{T:3.1}  only guarantees the existence of a subnet, instead of a subsequence, for which the representation formula \eqref{young} holds. The following result remedy this situation at least in the case where $u_\ve$ is a bounded sequence in $L^2(\Om)$, and the function $\Phi(z,x,\l)$ belongs to a reasonably large class which includes the particular case where $\Phi$ is linear in $\l$.  The latter then implies the basic result on two-scale convergence in the context of general algebras w.m.v., proved by  Casado-Diaz and Gayte, in \cite{MR1987520}. 

Given any Banach space $E$, we denote by $C_S(\R^m; E)$ the space of $E$-valued bounded continuous functions $\z(\cdot)$, such that $\z(t\l)\to\z_*(\l/|\l|)\in C(\SS^{m-1};E)$, as $t\to\infty$, where $\SS^{m-1}$ is the unity sphere in $\R^m$.  

\begin{theorem} \label{T:3.2} Let $\{u_\ve(x)\}_{\ve>0}$ be a sequence uniformly bounded in $L^2(\Om;\R^m)$, $\AA$ be an algebra w.m.v.\ in $\R^n$, and $\KK$ be the associated compact space.  Then, there exist a subsequence $u_{\ve_i}$, $i\in\N$,  where $\ve_i\downarrow0$ as $i\to\infty$, and a family of probability measures on 
$\R^m$, $\{\nu_{z,x}\}_{z\in {\mathcal K}, x\in \Om}$, weakly measurable with respect to the product of the Borel $\sigma$-algebras in ${\mathcal K}$ and $\R^n$,
such that
\begin{equation}\label{young'}
\lim_{i\to\infty}\int_{\Om}\Phi(\frac{x}{\ve_i},x,u_{\ve_i}(x))\,dx=
\int_{\Om}\int_{{\mathcal K}}\la\nu_{z,x},\Phi(\cdot,z,x)\ra\,d{\mathfrak m}(z)\,dx,
\end{equation}
for all $\Phi$ satisfying $\Phi(z,x,\l)/(1+|\l|)\in C_S(\R^m;\BB^2(\R^n;L^2(\Om)))$, and, for all $\l\in\R^m$, $\Phi(\cdot,\cdot,\l)$ is a regular test function.  In particular, $u_{\ve_i}$ two-scale converges to 
$U$, given by
\begin{equation}\label{eT3.2}
 U(z,x)=\la \nu_{z,x}, \Id(\cdot)\ra:=\int_{\R^m}\l\,d\nu_{z,x}(\l),\quad\text{for ($d\mm(z)\,dx$-)a.e.\ $(z,x)\in\KK\X\Om$},
 \end{equation} 
where  $\Id:\R^m\to\R^m$ denotes the identity mapping $Id(\l)=\l$.
and the limit in \eqref{eD3.1} keeps holding with $\varphi(z)\phi(x)$ replaced by any regular test function $\psi(z,x)\in \BB^2(\R^n;L^2(\Om))$. 
\end{theorem}

\begin{proof} Let $\FF=\{e_\l(z)\}_{\l\in\Lambda}$, be an orthonormal system in  $\BB^2(\R^n)\cap L_\loc^2(\R^n)$, whose space of finite linear combinations is dense in  
$\BB^2(\R^n)$.  Here, 
$\Lambda$ is any index set, and the existence of such a system may be obtained by using Zorn's Lemma, for instance. 
Let $\DD$ be a countable dense subset of $C_0^\infty(\Om)$, and fix $\phi\in\DD$. We will first prove the final part of the statement, concerning two-scale convergence.

Given any countable subset $\FF_0$ of $\FF$, we may find a  sequence $\ve_i\downarrow0$,
as $i\to\infty$, such that the limit 
\begin{equation}\label{e3.3}
\la L_\phi, e_\l\ra:= \lim_{i\to\infty}\frac1{|\Om|}\int_{\Om}u_{\ve_i}(x)\phi(x)e_\l(\ve_i^{-1}x)\,dx,
\end{equation}
exists for all $e_\l\in\FF_0$. Given $e_{\l_1},\dots,e_{\l_k}\in\FF_0$, by Bessel's inequality, we get
\begin{equation}\label{e3.4}
 |\la L_\phi,e_{\l_1}\ra|^2+\cdots+|\la L_\phi,e_{\l_k}\ra|^2\le C_\phi:=\|\phi\|_\infty^2\sup_{\ve>0}\frac1{|\Om|}\int_{\Om}|u_\ve|^2\,dx.
 \end{equation}
 Indeed, for each $i\in\N$, we consider the set $\{\tilde e_{\l_1,\ve_i},\cdots,\tilde e_{\l_k,\ve_i}\}$, orthonormal in $L^2(\Om; dx/|\Om|)$,  obtained by the Gram-Schmidt process from 
 $\{e_{\l_1,\ve_i}:=e_{\l_1}(\ve_i^{-1}\,\cdot),\cdots,  e_{\l_k,\ve_i}:=e_{\l_k}(\ve_i^{-1}\,\cdot)\}$.  We have, from Bessel's inequality,
\begin{equation}\label{e3.4'}
 |\la u_{\ve_i}\phi,\tilde e_{\l_1,\ve_i}\ra|^2+\cdots+|\la u_{\ve_i}\phi,\tilde e_{\l_k,\ve_i}\ra|^2\le \frac1{|\Om|}\int_{\Om}|u_{\ve_i}\phi|^2\,dx\le C_\phi,
 \end{equation} 
 where $\la\cdot,\cdot\ra$ denotes the scalar product of $L^2(\Om;\,dx/|\Om|)$. Since $|e_{\l_1,\ve_i}-\tilde e_{\l_1,\ve_i}|^2+\cdots+|e_{\l_k,\ve_i}-\tilde e_{\l_k,\ve_i}|^2\to0$ in 
 $L^1(\Om; \,dx/|\Om|)$, as $i\to\infty$, we get \eqref{e3.4} from \eqref{e3.4'}, in the limit as $i\to\infty$. 
 
 Therefore, from any such $\FF_0$, we may obtain a countable set $\FF_*$, with $\FF_0\subset\FF_*\subset\FF$,   with the property that, for some subsequence of $\ve_i$, still denoted $\ve_i$, the limit in \eqref{e3.3} exists for all $e_\l\in\FF_*$, and it is maximal for those $e_\l\in\FF_*$ such that $|\la L_\phi,e_{\l}\ra|>0$. The latter means that we cannot enlarge the subset of $e_\l\in\FF_*$ such that $|\la L_\phi,e_{\l}\ra|>0$, by extracting a further subsequence. Indeed, we can obtain $\FF_*$ in the following way. If there is any 
 $e_\l\in\FF\setminus\FF_0$ such that 
 \begin{equation}\label{e3.5}
 \limsup_{i\to\infty}\left|\frac1{|\Om|}\int_{\Om}u_{\ve_i}(x)\phi(x)e_\l(\ve_i^{-1}x)\,dx\right|>1,
 \end{equation}
 we can take the new countable family $\FF_{0,1}:=\FF_0\cup\{e_\l\}$ and a new subsequence $\{\ve_i\}$ for which the limit in \eqref{e3.3} exists for all $e_\l\in\FF_{0,1}$.   Now, if there is $e_\mu\in \FF\setminus\FF_{0,1}$ such that \eqref{e3.5} holds, for the new subsequence $\{\ve_i\}$, we can proceed in the same way and obtain  a subsequence from the last subsequence $\{\ve_i\}$ and define a new countable family $\FF_{0,2}=\FF_{0,1}\cup\{e_\mu\}$. We can do that only a finite number of times because of \eqref{e3.4}.  In this way we obtain a countable family $\FF_{0,N_0}$ and a subsequence of $\{\ve_i\}$, which we keep denoting $\ve_i$, maximal with respect to the property \eqref{e3.5}. We then proceed in an analogous way looking for $e_\l\in\FF\setminus \FF_{0,N_0}$ such that
 \begin{equation}\label{e3.6}
 \limsup_{i\to\infty}\left|\frac1{|\Om|}\int_{\Om}u_{\ve_i}(x)\phi(x)e_\l(\ve_i^{-1}x)\,dx\right|>\frac12,
 \end{equation}
and get a new subsequence and a countable family $\FF_{0,N_1}$,  maximal  with respect to the property \eqref{e3.6}. In this way, we extract subsequences for which the limits
\eqref{e3.3} exist, and define countable families $\FF_{0,N_k}$ which are maximal with respect to its elements  satisfying 
 \begin{equation}\label{e3.7}
 \limsup_{i\to\infty}\left|\frac1{|\Om|}\int_{\Om}u_{\ve_i}(x)\phi(x)e_\l(\ve_i^{-1}x)\,dx\right|>\frac1{2^k},
 \end{equation} 
 By a diagonal process, we obtain a subsequence still denoted by $\{\ve_i\}$ and a countable family $\FF_*\subset\FF$, which is maximal for elements satisfying \eqref{e3.7}, for any $k\in\N$. In particular, 
 for any $e_\l\in\FF\setminus\FF_*$, we have 
 \begin{equation}\label{e3.8}
 \limsup_{i\to\infty}\left|\frac1{|\Om|}\int_{\Om}u_{\ve_i}(x)\phi(x)e_\l(\ve_i^{-1}x)\,dx\right|=0.
 \end{equation} 
 
Now, let $\HH$ be the closure in $\BB^2$ of the space spanned by $\FF_*$. In particular, $\HH$ is separable. Let $\HH_0$ be the subspace of $\HH$ spanned by $\FF_*$. 
Given $\phi\in\DD$, we may find a subsequence $\ve_i\downarrow0$, as $i\to\infty$, $i\in\N$, for which the following limit exists for all $g\in\FF_*$, and, hence, for all $g\in\HH_0$,
$$
\la L_\phi, g\ra :=\lim_{i\to\infty}\,\la u_{\ve_i}\phi, g_{\ve_i}\ra,\quad\text{with $g_\ve(x)=g(\ve^{-1} x)$}.
$$  
Clearly, $L_\phi$ is  a continuous linear functional over $\HH_0$, and, by density, it may be extended to $\HH$  . Since $\DD$ is countable, there is a subsequence of $\ve_i$, 
which we still label $\ve_i$, such that the following limit exists
\begin{equation}\label{e3.9}
\la L_\phi, g\ra :=\lim_{i\to\infty}\,\la u_{\ve_i}\phi, g_{\ve_i}\ra,\quad\text{for all $g\in\HH$ and all $\phi\in\DD$}.
\end{equation}
We will henceforth, with no loss of rigor, ignore the normalization factor $1/|\Om|$, and simply interpret the scalar product on the right-hand side of \eqref{e3.9} as exactly the one of $L^2(\Om)$.

Now, given any $g\in\BB^2$, we may write $g=g^\HH+\rho$, with $g^\HH\in\HH$ and $\rho\in\HH^\perp$, with equality holding in the sense of $\BB^2$. By \eqref{e3.8}, we deduce that the limit on the right-hand side  of \eqref{e3.9} exists and equals 0, if, instead of $g\in\HH$, we take $\rho\in\HH^\perp$. Therefore, the limit on the right-hand side of \eqref{e3.9} exists for all $g\in\BB^2$, and we may define 
\begin{equation}\label{e3.10}
\la L_\phi, g\ra :=\lim_{i\to\infty}\,\la u_{\ve_i}\phi, g_{\ve_i}\ra,\quad\text{for all $g\in\BB^2$ and all $\phi\in C_0^\infty(\Om)$}.
\end{equation}
Now, from the sequence $\ve_i$, we may then apply the Theorem~\ref{T:3.1}, and obtain a subnet $\ve_j(d)$, indexed by some directed set $D$, and a family of two-scale Young-measures $\nu_{z,x}$ so that the representation formula  \eqref{young} is verified.  In particular, for all $\varphi\in\AA$, $\phi\in C_0^\infty(\Om)$, we must have
\begin{equation}\label{e3.11}
\la L_\psi,\varphi\ra=\int_{\Om}\int_{\KK}\la\nu_{z,x}, \Id(\cdot)\ra \varphi(z)\phi(x)\,d\mm(z)\,dx.
\end{equation}
Therefore, given any such  weakly measurable family of two-scale Young measures $\nu_{z,x}$, obtained from an application of the Theorem~\ref{T:3.1} to the sequence 
$\{u_{\ve_i}\}$, we then define 
\begin{equation}\label{e3.12}
U(z,x):=\la \nu_{z,x},\Id(\cdot)\ra.
\end{equation}
By \eqref{e3.7}, we deduce that $U(z,x)$ satisfies \eqref{eD3.1} for the subsequence $\{u_{\ve_i}\}$.   

The fact that $U\in\BB^2(\R^n;L^2(\Om))$ follows from  the definition in \eqref{e3.12}. Indeed, we have the following
\begin{align*}
\int_{\Om}\int_{\KK}|U(z,x)|^2\,d\mm(z)\,dx &\le \int_{\Om}\int_{\KK}\la\nu_{z,x},|\l|^2\ra\,d\mm(z)\,dx\\
                            &= \lim_{i\to\infty}  \int_{\Om}\int_{\KK}\la\nu_{z,x},\min\{|\l|^2,i\}\ra\,d\mm(z)\,dx\\
                            &\le \limsup_{\ve\to0}\int_{\Om}|u_\ve|^2\,dx,
                            \end{align*}
and, thus, the assertion follows from the uniform boundedness of $u_\ve$ in $L^2(\Om;\R^m)$.        

Finally, we prove the validity of \eqref{eT3.2} when $\varphi(z)\phi(x)$ is replaced by a regular test function $\psi\in\BB^2(\R^n;L^2(\Om))$.  This may be verified by first observing that  the corresponding relation holds for $\psi\in \AA(\R^;C_0(\Om))$. Indeed, functions in $\AA(\R^;C_0(\Om))$ may be approximated in the $\sup$-norm by finite linear combinations of functions of the form $\varphi(z)\phi(x)$, with $\varphi\in\AA$ and $\phi\in C_0^\infty(\Om)$, which can be seen, for instance, through the use of a partition of unity in $\KK$ and by the identification of  $C(\KK;C_0(\Om))$ with $\AA(\R^n; C_0(\Om))$. Next, from $\AA(\R^n;C_0(\Om))$, the formula \eqref{eT3.2} can be extended to functions in 
$\AA(\R^n;C(\bar\Om))$. Indeed, given $\psi\in\AA(\R^n;C(\bar\Om))$, let $\xi\in C_0^\infty(\Om)$ be such that $0\le \xi\le1$ and $\xi(x)=1$ for $x\in\tilde\Om\Subset\Om$, with
$|\Om\setminus\tilde\Om|$ so small as we wish. Then, 
 $$
 \sup_{\ve>0}\left|\int_{\Om} u_\ve(x)\psi(\frac{x}{\ve},x)(1-\xi(x))\,dx\right|\le \|\psi\|_\infty\sup_{\ve>0}\|u_\ve\|_{L^2(\Om)}|\Om\setminus\tilde\Om|^{1/2},
$$
and since the formula holds for $\psi(z,x)\xi(x)$, we can easily conclude the validity of \eqref{eT3.2}, if $\varphi(z)\phi(x)$ is replaced by any $\psi\in\AA(\R^n;C(\bar\Om))$. 
Further, we claim that we may extend the validity of \eqref{eT3.2}, now, from test functions in $\AA(\R^n;C(\bar\Om))$ to test functions in $\BB^2(\R^n; C(\bar\Om))$. Indeed,
for any $\psi\in\BB^2(\R^n;C(\bar\Om))$, we may find $\tilde\psi\in\AA(\R^n;C(\bar\Om))$, so that $\rho(z):=\|\psi(z,\cdot)-\tilde\psi(z,\cdot)\|_{C(\bar\Om)}$ is so small as we wish in the norm of $\BB^2(\R^n)$. Therefore, from
 $$
 \limsup_{\ve\to0}\left|\int_{\Om} u_\ve(x)(\psi(\frac{x}{\ve},x)-\tilde\psi(\frac{x}{\ve},x))\,dx\right|\le \|\rho\|_{\BB^2(\R^n)}\sup_{\ve>0}\|u_\ve\|_{L^2(\Om)},
$$
we deduce the desired extension. Now, let $\psi\in\BB^2(\R^n;L^2(\Om))$ be a regular test function, and define
$$
\psi_\d(z,x):=\Medint_{B_\d(x)} \psi(z,y)\,dy,
$$    
where we extend $\psi$ as 0 for $x\in\R^n\setminus\Om$. Then, $\psi_\d\in \BB^2(\R^n; C(\bar\Om))$ and, so, \eqref{eT3.2} holds for $\psi_\d(z,x)$ replacing $\varphi(z)\phi(x)$, as we have just proved.  We then  get  the validity of \eqref{eT3.2} with  $\psi$ as test function,  by sending  $\d\to0$, using \eqref{eT3.2-2}.                     
                            
Now, we observe that, if $\Phi(z,x,\l)/(1+|\l|) \in C_S(\R^m;\BB^2(\R^n;L^2(\Om)))$, and for all $\l\in\R^m$, $\tilde\varphi(\cdot,\cdot,\l)$ is a regular test function, then 
$\Phi$ may be approximated in $C_b(\R^m;\BB^2(\R^n;L^2(\Om)))$ by finite linear combinations of functions of the form $g(\l)\psi(z,x)$ with $g(\l)/(1+|\l|)\in C_S(\R^m)$, and $\psi(z,x)$ a regular test function. Since $C_S(\R^m)$ is a separable space, we may restart the same procedure used above for $u_\ve\phi$, with $\phi\in D$, this time for 
$g(u_{\ve})\phi$, for $g(\l)/(1+|\l|)$ belonging to a countable dense subset of $C_S(\R^m)$, and $\phi\in D$. Taking a diagonal subsequence good for all such $g$ and $\phi$, we can conclude the proof of \eqref{young'} proceeding as we did for proving that the limit  in \eqref{eD3.1} keeps holding with $\varphi(z)\phi(x)$ replaced by any regular test function $\psi(z,x)\in \BB^2(\R^n;L^2(\Om))$.  Putting together all that has been proved, we arrive at the validity of \eqref{young'}, finishing the proof.

\end{proof}

\begin{remark}\label{R:3.2} In the subsequent discussion we will be mainly concerned with sequences of functions $v_\ve(x,t)$, with $(x,t)\in\Om\X[0,T)$, where $x$ will always denote the space variable and $t$ the time variable. The homogenization processes to be considered will always refer to the space variable, while the time variable may be seen as a parameter, or simply as a macroscopic coordinate which does not take part in the homogenization process. Both Theorem~\ref{T:3.1} and Theorem~\ref{T:3.2}   possess obvious analogues suitable for sequences of functions depending on $(x,t)$ for which only the variable $x$ takes part in the homogenization process.
\end{remark}

\section{Convexity, monotonicity and fundamental results}\label{S:4}

In this section we establish two  important theorems which will be the two main tools for the homogenization analysis carried out in the next section. 

Let $U\subset \R^m$ be an open convex set. We recall that $G: U\to\R$ is said to be convex if $G(\theta x+(1-\theta) y)\le \theta G(x)+(1-\theta) G(y)$, for all 
$\theta\in[0,1]$, $x,y\in U$, and it is said to be {\em strictly} convex if in the last inequality we have $<$, instead of $\le$, if $0<\theta<1$, whenever $x\ne y$. The subdifferential $\po G(x)$, of $G$ at $x$, is the set of $\l\in \R^m$ for which 
\begin{equation}\label{esub}
G(y)-G(x)\ge \l\cdot (y-x),
\end{equation}
for all $y\in U$. 

Let $\Psi:\R^n\X\Om\X\R^m\to\R$ be a function satisfying the following conditions:

\begin{enumerate}

\item[\bf{($\mathbf{\Psi 1}$)}] For all $\l\in\R^m$, $\Psi(\cdot,\cdot,\l)\in \BB^2(\R^n;L^2(\Om))$ is a regular test function ({\em cf.} Definition~\ref{D:3.2});
\medskip

\item[\bf{($\mathbf{ \Psi 2}$)}] $\Psi(z,x,\cdot)$ is strictly convex for $\,d\mm(z)\,dx$-a.e.\ $(z,x)\in\KK\X\Om$;
\medskip

\item[\bf{ ($\mathbf{\Psi3}$)}]  There exists $c>0$ and $h\in\BB^2(\R^n;L^2(\Om))$, regular test function, such that for a.e.\ $(z,x)\in\R^n\X\Om$, 
\begin{equation}\label{ePsi3}
|\Psi(z,x,\l)-\Psi(z,x,\mu)|\le |\l-\mu|\left\{ c\max \{|\l|\,,\,|\mu|\} + h(z,x)\right\}, \ \text{for all $\l,\mu\in\R^m$};
\end{equation}
\medskip

\item[\bf{($\mathbf{\Psi4}$)}] There exists $\tilde c>0$, $W, \tilde h\in\BB^2(\R^n; L^2(\Om))$, regular test functions,  such that for a.e.\ $(z,x)\in\R^n\X\Om$, 
\begin{equation}\label{ePsi4}
\Psi(z,x,\l)\ge p(z,x,\l):=\tilde c|\l|^2+W(z,x)\cdot\l+\tilde h (z,x).
\end{equation}

\end{enumerate}

We now state and prove the  theorem which is the first  of the two main tools for the homogenization analysis developed in the next section. Here, $H^1(\Om)$
 denotes, as usual, the Sobolev space of the functions in $L^2(\Om)$ whose first order distributional derivatives are in $L^2(\Om)$, endowed with its
canonic Hilbert space structure, $H_0^1(\Om)$ is the closure of $C_0^\infty(\Om)$ in $H^1(\Om)$, and $H^{-1}(\Om)$ is the dual of $H_0^1(\Om)$.  

\begin{theorem}\label{T:4.1} Suppose $\Psi(z,x,\l)$ satisfies {\bf($\mathbf{\Psi1}$)}--{\bf($\mathbf{\Psi4}$)}. Let  $\{u_\ve(x,t)\}_{\ve>0}$ and $\{w_\ve(x,t)\}_{\ve>0}$ satisfy:

\begin{enumerate}

\item[(i)] $u_\ve:\Om\X[0,T)\to \R^m$ is uniformly bounded in $L^2(0,T;H_0^1(\Om))$;
\medskip

\item[(ii)]  $w_\ve:\Om\X[0,T)\to\R^m$ is uniformly bounded in $H^1(0,T;H^{-1}(\Om))\bigcap L^2(0,T;L^2(\Om))$;
\medskip

\item[(iii)]  $w_\ve(x,t)\in \po\Psi(\frac{x}{\ve},x, u_\ve(x,t))$ for a.e.\ $(x,t)\in\Om\X[0,T)$. 
\end{enumerate}

Then, by passing to a subsequence if necessary, we have $u_\ve\to u$ in $L^p(\Om\X[0,T))$, for any $1\le p<2$, and $w_\ve\wto  w$ in $L^2(\Om\X[0,T))$, such that
\begin{equation}\label{eT4.1}
 w(x,t)\in\po \mv\Psi(x,u(x,t)), \ \text{for a.e.\ $(x,t)\in\Om\X[0,T)$, with}\  \mv\Psi(x,\l):=\int_{\KK}\Psi(z,x,\l)\,d\mm(z).
\end{equation}
   
\end{theorem}

\begin{proof} We carry out the proof in several steps.

{\em Step 1}. Let $\nu_{z,x,t}$ be a two-scale Young measure generated by a subnet of $u_\ve(x,t)$, with respect to the algebra w.m.v.\ $\AA$, where the homogenization process involves only the space variable $x$. We claim that $\nu_{z,x,t}$ is independent of $z\in\KK$;  more precisely, we have
\begin{equation}\label{e4.1}
\nu_{z,x,t}=\bar\nu_{x,t}:=\int_{\KK}\nu_{z,x,t}\,d\mm(z),\quad\text{for a.e.\ $(z,x,t)\in\KK\X\Om\X[0,T)$}.
\end{equation}
Indeed, let $\theta(u)\in C_0^\infty(\R^m)$, $\varphi(x,t)\in C_0^\infty(\Om\X(0,T))$, and $g(z)\in\AA$ such that the derivatives up to second order of $g$ also belong to $\AA$. We then have
\begin{multline}\label{e4.2}
\int_{\Om\X(0,T)}\ve\, \nabla\theta(u_\ve)\cdot\nabla_z g(\frac{x}{\ve})\varphi(x,t)\,dx\,dt\\
=-\int_{\Om\X(0,T)}  \theta(u_\ve(x,t))\left\{\ve\,\nabla_z g(\frac{x}{\ve})\cdot\nabla\varphi(x,t)+\Delta_z g(\frac{x}{\ve}) \varphi(x,t)\right\}\,dx\,dt.
\end{multline}

Taking the limit in \eqref{e4.2}, for a suitable subnet of $u_\ve$,  we obtain
\begin{equation}\label{e4.2'}
\int_{\Om\X(0,T)}\int_{\KK} \la\nu_{z,x,t}, \theta\ra \Delta g(z) \varphi(x,t)\,d\mm(z)\,dx\,dt=0.
\end{equation}
Therefore, by Lemma~\ref{Lapla}, for a.e.\ $(x,t)\in\Om\X(0,T)$, we have that  $\la\nu_{z,x,t},\theta\ra=\la \bar\nu_{x,t},\theta\ra$, for $\mm$-a.e.\ $z\in\KK$, for all $\theta\in C_0^\infty(\R^m)$, and the claim follows by a trivial approximation argument.     

{\em Step 2}. We claim that, for any subnet of $u_\ve$ generating $\nu_{z,x,t}$,  the the following  representation formula holds for 
$\Psi\in \BB^2\left(\R^n; L^2(\Om; C_b(\R^m))\right)$,
satisfying {\bf($\mathbf{\Psi1}$)}--{\bf($\mathbf{\Psi4}$)}:
\begin{equation}\label{e4.3}
\lim_{\ve\to 0} \int_{\Om\X(0,T)} \Psi(\frac{x}{\ve},x, u_\ve(x,t))\,dx\,dt=\int_{\Om\X(0,T)}\int_{\KK} \la\nu_{z,x,t},\Psi(z,x,\cdot)\ra\,d\mm(z)\,dx\,dt.
\end{equation}
Indeed, we first observe that for $\Psi$ allowing a decomposition of the form $\Psi(z,x,\l)=G(z,x)\theta (\l)$, where $G\in\BB^2(\R^n;L^2(\Om))$ is a regular test function, \eqref{e4.3} holds. This may be proved by first observing that  the corresponding relation holds for $\Psi_\d(z,x,\l)=G_\d(z,x)\theta(\l)$, with
$$
G_\d(z,x):=\Medint_{B_\d(x)} G(z,y)\,dy,
$$    
observing that $G_\d\in\BB^2(\R^n;C(\bar\Om))$ (we may extend $G$ as 0 for $x\in\R^n\setminus\Om$), and then  getting  the validity of \eqref{e4.3} for $\Psi$ allowing the referred decomposition, by sending $\d\to0$, using \eqref{eT3.2-2}.  
Next, for a general $\Psi$ satisfying  {\bf($\mathbf{\Psi1}$)}--{\bf($\mathbf{\Psi4}$)}, we first observe that, if $\z_\d$ is an approximation of the identity sequence in $\R^m$, for each fixed $\mu\in\R^m$, 
$\Psi(z,x,\mu)\z_\d(\l-\mu)$ verifies the decomposition just mentioned. Now, by using \eqref{ePsi3}, we may obtain the validity of \eqref{e4.3} for
$$
\Psi_\d(z,x,\l):= \int_{\R^m} \Psi(z,x,\mu)\z_\d(\l-\mu)\,d\mu,
$$
for any $\d>0$. Indeed, let $\xi_R(\l)\in C_0^\infty(\R^n)$ be such that $0\le \xi_R\le 1$, $\xi_R(\l)=1$, if $|\l|<R$. Then, for any $\g>0$, we may choose $R>0$ large enough so that
\begin{equation}\label{e4.3'}
\left| \int_{\Om\X(0,T)} \Psi_\d(\frac{x}{\ve},x, u_\ve(x,t))(1-\xi_R(u_\ve(x,t)))\,dx\,dt\right|<\g, \quad \text{uniformly in $\ve>0$ and $\d>0$},
\end{equation}
by the uniform boundedness of $u_\ve$ is $L^2(\Om\X(0,T)$. 

 and we may use \eqref{ePsi3} again to send $\d\to0$, proving the claim. 

{\em Step 3}. We claim that, for $\Psi(z,x,\l)$ satisfying {\bf($\mathbf{\Psi1}$)}--{\bf($\mathbf{\Psi4}$)}  and for any subnet of $u_\ve(x,t)$ generating a two-scale Young measure 
$\nu_{z,x,t}$,  for all $ \varphi\in C_0^\infty(\Om\X(0,T))$, $\varphi\ge0$,  the following relation holds:  
\begin{equation}\label{e4.4}
\liminf_{\ve\to 0} \int\limits_{\Om\X(0,T)} \Psi(\frac{x}{\ve},x, u_\ve(x,t))\varphi(x,t)\,dx\,dt\ge\int\limits_{\Om\X(0,T)}\int_{\KK} \la\nu_{z,x,t},\Psi(z,x,\cdot)\ra \varphi(x,t)\,d\mm(z)\,dx\,dt.
\end{equation}
Indeed, for $i\in\N$, let
$$
\Psi_i(z,x,\l):=\min\{i,\ \Psi(z,x,\l)\}.
$$
Observe that $\Psi_i\in\BB^2(\R^n;L^2(\Om; C_b(\R^m)))$ and satisfies the conditions {\bf($\mathbf{\Psi1}$)} and {\bf($\mathbf{\Psi3}$)}.  Now, by \eqref{ePsi4}, we  have
$$
\Psi(z,x,\l)\ge \inf_{\mu\in\R^m} \Psi(z,x,\mu)\ge \inf_{\mu\in\R^m} p(z,x,\mu):=-\frac{|W(z,x)|^2}{2\tilde c}+\tilde h(z,x).
$$
Therefore, $\Psi$ is bounded below by a function which is integrable with respect to $d\mm(z)\,dx\,dt$. Since $\Psi_i(z,x,\l)\uparrow \Psi(z,x,\l)$, from the monotone convergence theorem, applied first for $\nu_{z,x,t}$, and then for $\,d\mm(z)\,dx\,dt$, we then obtain
\begin{align*}
&\int\limits_{\Om\X(0,T)}\int_{\KK} \la\nu_{z,x,t}, \Psi(z,x,\cdot)\ra \varphi(x,t)\,d\mm(z)\,dx\,dt=\lim_{i\to\infty}\int\limits_{\Om\X(0,T)}\int_{\KK}\la\nu_{z,x,t},\Psi_i(z,x,\cdot)\ra\varphi(x,t)\,d\mm(z)\,dx\,dt\\
&\qquad=\lim_{i\to\infty}\lim_{\ve\to0}\int\limits_{\Om\X(0,T)}\Psi_i(\frac{x}{\ve},x, u_\ve(x,t))\varphi(x,t)\,dx\,dt\le \liminf_{\ve\to0}\int\limits_{\Om\X(0,T)}\Psi(\frac{x}{\ve},x, u_{\ve}(x,t))\varphi(x,t)\,dx\,dt,
\end{align*}
which proves the claim.

{\em Step 4}. Next, we claim that, for any subnet of $u_\ve(x,t)$, converging weakly  to $u(x,t)$ in $L^2(0,T; H^1(\Om))$, whose existence is guaranteed by Banach-Alaoglu theorem, and generating a two-scale Young measure $\nu_{z,x,t}$, for
$\Psi(z,x,\l)$ satisfying  {\bf($\mathbf{\Psi1}$)}--{\bf($\mathbf{\Psi4}$)},  and $ \varphi\in C_0^\infty(\Om\X(0,T))$, $\varphi\ge0$, we have
\begin{equation}\label{e4.5}
\limsup_{\ve\to0} \int\limits_{\Om\X(0,T)} \Psi(\frac{x}{\ve},x,u_\ve(x,t))\varphi(x,t) \,dx\,dt\le \int\limits_{\Om\X(0,T)}\int_{\KK} \Psi(z,x,u(x,t))\varphi(x,t)\,d\mm(z)\,dx\,dt.
\end{equation}
Indeed, by the hypothesis (ii) in the statement and Aubin's lemma (see, e.g., \cite{MR0259693}), we deduce that $w_\ve\to  w$ in $L^2(0,T;H^{-1}(\Om))$. We then have, for all 
$\varphi\in C_0^\infty(\Om\X(0,T))$, 
\begin{equation}\label{e4.6}
\lim_{\ve\to0}\int\limits_{\Om\X(0,T)} w_\ve(x,t)\cdot u_\ve(x,t)\varphi(x,t)\,dx\,dt=\int\limits_{\Om\X(0,T)} u(x,t)\cdot w(x,t)\varphi(x,t)\,dx\,dt.
\end{equation}
Now, since $w_\ve(x,t)\in\po\Psi(\frac{x}{\ve},x,u_\ve(x,t))$, a.e.\ in $\Om\X(0,T)$,  for $\varphi\ge0$, we have
\begin{equation}\label{e4.7}
\int\limits_{\Om\X(0,T)} (\Psi(\frac{x}{\ve},x,u_\ve(x,t))-\Psi(\frac{x}{\ve}, x,u(x,t)))\varphi(x,t)\,dx\,dt\le \int\limits_{\Om\X(0,T)} w_\ve(x,t)\cdot(u_\ve(x,t)-u(x,t))\varphi(x,t)\,dx\,dt .
\end{equation}
Taking the $\limsup$ on  \eqref{e4.7} and using \eqref{e4.6}, we then deduce \eqref{e4.5}, proving the claim.

{\em Step 5}. We claim that $u_\ve\to u$ in $L^p(\Om\X[0,T))$, for any $1\le p<2$.

Indeed, from \eqref{e4.5}, \eqref{e4.1}, Jensen inequality, and \eqref{e4.4}, we obtain
\begin{equation}\label{e4.8}
\begin{aligned}
& \limsup_{\ve\to0} \int\limits_{\Om\X(0,T)} \Psi(\frac{x}{\ve},x,u_\ve(x,t))\varphi(x,t) \,dx\,dt\le \int\limits_{\Om\X(0,T)}\int_{\KK} \Psi(z,x,u(x,t))\varphi(x,t)\,d\mm(z)\,dx\,dt\\
&\le\int\limits_{\Om\X(0,T)}\int_{\KK} \la \bar\nu_{x,t} , \Psi(z,x,\cdot )\ra \varphi(x,t)\,d\mm(z)\,dx\,dt\le \liminf_{\ve\to 0} \int\limits_{\Om\X(0,T)} \Psi(\frac{x}{\ve},x, u_\ve(x,t))\varphi(x,t)\,dx\,dt,
\end{aligned}
\end{equation}
from which it follows
\begin{equation} \label{e4.9}
\int\limits_{\Om\X(0,T)} \la\bar\nu_{x,t},\mv \Psi(x,\cdot)\ra\,\varphi(x,t)\,dx\,dt=\int\limits_{\Om\X(0,T)}\mv \Psi(x,u(x,t))\varphi(x,t)\,dx\,dt.
\end{equation}
Since $\mv\Psi (x,\cdot)$ is strictly convex, and $\varphi\in C_0^\infty(\Om\X(0,T))$ is arbitrary, we conclude that $\ \bar\nu_{x,t}=\d_{u(x,t)}$,
and so, by a well known fact on classical Young measures ({\em cf.}, e.g.,  \cite{MR584398}),  and the uniform boundedness of $u_\ve$ in $L^2(\Om\X(0,T))$, we conclude that, by passing to a suitable subsequence, 
$u_\ve\to u$ in $L^p(\Om\X[0,T))$, for any $1\le p<2$, as asserted.

{\em Step 6}. Finally, we claim that \eqref{eT4.1} holds. 

Indeed, since $w_\ve(x,t)\in\po\Psi(\frac{x}{\ve}, x, u_\ve(x,t))$, a.e.\ in $\Om\X(0,T)$, for any $\l\in\R^m$, $\varphi\in C_0^\infty(\Om\X(0,T))$, $\varphi\ge0$,  we have
\begin{equation}\label{e4.10}
\int\limits_{\Om\X(0,T)} \left(\Psi(\frac{x}{\ve}, x,\l)-\Psi(\frac{x}{\ve},x,u_\ve(x,t))\right)\varphi(x,t)\,dx\,dt\ge \int\limits_{\Om\X(0,T)} w_\ve(x,t)\cdot(\l-u_\ve(x,t))\varphi(x,t)\,dx\,dt.
\end{equation}
Taking the limit as $\ve\to0$ in \eqref{e4.10}, using \eqref{e4.6} and 
\begin{equation} \label{e4.11}
 \lim_{\ve\to0} \int\limits_{\Om\X(0,T)} \Psi(\frac{x}{\ve},x,u_\ve(x,t))\varphi(x,t) \,dx\,dt= \int\limits_{\Om\X(0,T)}\mv \Psi(x,u(x,t))\varphi(x,t)\,dx\,dt,
\end{equation}
which follows immediately from \eqref{e4.8}, we get
\begin{equation}\label{e4.12}
\int\limits_{\Om\X(0,T)} \left(\mv\Psi( x,\l)-\mv\Psi(x,u(x,t))\right)\varphi(x,t)\,dx\,dt\ge \int\limits_{\Om\X(0,T)}  w(x,t)\cdot(\l-u(x,t))\varphi(x,t)\,dx\,dt,
\end{equation}
from which \eqref{eT4.1} follows, finishing the proof.

 \end{proof}
\bigskip

We next prepare our way to establishing our second main tool, which will be concerned with monotonicity. 
The following lemma from  \cite{MR1987520}  extends to the context of ergodic algebras the analogous result for two-scale convergence in the periodic case,
first established in \cite{MR1185639}. For the sake of completeness, we give here a short proof that makes use of Lemmas~\ref{L:ZK} and~\ref{L:ZK0}. 

\begin{lemma}\label{L:4.1}  Let $\AA$ be an ergodic algebra in $\R^n$, and let $\{u_\ve\}$ be uniformly bounded in $ L^2(0,T; H^1(\Om))$. Then, there exists a subsequence 
$\ve_i\downarrow0$, as $i\to\infty$, and a function $u_1\in \Nu_{\pot}^2(\R^n; L^2(\Om\X(0,T)))$  such that $\nabla_x u_{\ve_i}$ two-scale converges to $\nabla_x u(x,t)+ u_1(z,x,t)$.
\end{lemma}

\begin{proof} Let $Z(A)$ be defined by \eqref{eZ(A)}. It is easy to verify that, if $f\in Z(\AA)$, then $f\in C^\infty(\R^n)$ and the equation $\Delta u=f$ has a solution $u\in Z(\AA)$ ({\em cf.} \cite{ MR1329546}). 
If $V\in\AA(\R^n;\R^n)$ is a vector field with components in $Z(\AA)$, we easily see that $\div V\in Z(\AA)$.  In particular,  there is a solution for $\Delta v=\div V$, with $v\in Z(\AA)$. 
Then, for such $V$, its orthogonal projection on $\Nu_{\sol}^2$, given by $V-\nabla v$,  is  also a vector field with components in $Z(\AA)$, so that such vector fields form a dense subspace of $\Nu_{\sol}^2$.  Let $u\in L^2(0,T; H^1(\Om))$ be the weak limit of a subsequence $u_{\ve_i}$ of $u_\ve$, whose existence is guaranteed by the boundedness of $u_\ve$ in $L^2(0,T; H^1(\Om))$. Since $\nabla_x u_{\ve_i}-\nabla_x u$ is bounded in $L^2(\Om\X(0,T))$, Theorem~\ref{T:3.2} implies that there is a subsequence of $\ve_i$, still denoted 
$\ve_i$, such that  $\nabla_x u_{\ve_i}-\nabla_x u$ two-scale converges to some $u_1\in  \BB^2(\R^n;L^2(\Om\X (0,T);\R^n))$. Now, for any $\phi\in C_0(\Om\X(0,T))$ and 
$V\in \Nu_{\sol}^2$, whose components are in $Z(\AA)$, we have
$$
\int\limits_{\Om\X(0,T)}(\nabla_x u_{\ve_i}-\nabla_x u)\cdot \phi(x,t)V(\frac{x}{\ve_i})\,dx\,dt=-\int\limits_{\Om\X(0,T)}( u_{\ve_i}- u)\nabla_x \phi(x,t)\cdot V(\frac{x}{\ve_i})\,dx\,dt.
$$
Taking the limit as $\ve_i\to0$, we obtain
$$
\int\limits_{\Om\X(0,T)}\int_{\KK} u_1(z,x,t)\cdot \phi(x,t) V(z)\,d\mm(z)\,dx\,dt=0.
$$
Since $V$ is an arbitrary element of a dense subspace of $\Nu_{\sol}^2$,  $\phi\in C_0^\infty(\Om)$ is arbitrary,  and, clearly, $\mv u_1(x,t)=0$, a.e.\ $(x,t)\in\Om\X(0,T)$,   we conclude that $u_1\in  \Nu_{\pot}^2(\R^n; L^2(\Om\X(0,T)))$, as asserted.
\end{proof}

\bigskip
  We now consider $\psi:\R^n\X\Om\X[0,T)\X\R\X\R^n\to\R$ satisfying the following conditions:
  
  \begin{enumerate}
  
   \item[{\bf($\mathbf{\psi1}$)}] For all $u\in\R$, $\eta \in\R^n$, $t\mapsto \psi(\cdot,\cdot, t, u,\eta)\in L^2\left(0,T; \BB^2(\R^n;L^2(\Om))\right)$,  and for a.e.\ $t\in[0,T)$,  $\psi(\cdot,\cdot, t, u,\eta)$
  is a regular test function;
  \medskip

  \item[{\bf($\mathbf{\psi2}$)}] $\psi(z,x,t,u,\cdot)$ is $C^1$ and convex in $\R^n$, for all $u\in\R$ and a.e.\ $(z,x,t)\in\R^n\X\Om\X(0,T)$;
  \medskip
  
  \item[{\bf($\mathbf{\psi3}$)}]  $\psi(z,x,t,\cdot,\eta)$ is continuous in $\R$, for all $\eta\in\R^n$ and  a.e.\ $(z,x,t)\in\R^n\X\Om\X(0,T)$;
  \medskip
  
\item[{\bf($\mathbf{\psi4}$)}] There exist $c_\psi>0$, $h_\psi\in L^2(\Om\X(0,T))$ such that, for all $u\in\R$, $\eta \in\R^n$, and  a.e.\ $(z,x,t)\in\R^n\X\Om\X(0,T)$ we have
  \begin{equation}\label{e4.14}
  \psi(z,x,t,u,\eta)\ge c_\psi|\eta|^2+h_\psi(x,t).
  \end{equation}
  \end{enumerate}
  
  We also consider 
  $$
  \a(z,x,t,u,\eta):=\nabla_\eta\psi(z,x,t,u,\eta), 
  $$
  for $\psi$ satisfying ($\psi1$)--($\psi4$), so that $\a(z,x,t,u,\cdot)$ is continuous and monotone, that is,
  \begin{equation}\label{ea0}
 \left( \a(z,x,t,u,\eta_1)-\a(z,x,t,u,\eta_2)\right)(\eta_1-\eta_2)\ge0, \quad \eta_1,\eta_2\in\R^n, 
  \end{equation}
  for all $u\in\R$, and  ($\,dz\,dx\,dt$)-a.e.\ $(z,x,t)\in\R^n\X\Om\X(0,T)$. We further assume:
  
  \begin{enumerate}
  
   \item[{\bf($\mathbf{\a1}$)}]  For all $u\in\R$, $\eta \in\R^n$, $t\mapsto\a(\cdot,\cdot, t, u,\eta)\in L^2\left(0,T; \BB^2(\R^n;L^2(\Om;\R^n))\right)$, and for a.e.\ $t\in[0,T)$,  each component of 
 $\a(\cdot,\cdot, t, u,\eta)$ is a regular test function;
  \medskip
  
 \item[{\bf($\mathbf{\a2}$)}]  $\a(z,x,t,\cdot,\eta)$ is continuous for all $\eta\in\R^n$, and a.e.\ $(z,x,t)\in\KK\X\Om\X(0,T)$;
 \medskip
 
 \item[{\bf($\mathbf{\a3}$)}] There exist $c_\a>0$, $h_\a\in L^1(\Om\X(0,T))$ such that for all $u\in\R$, $\eta\in\R^n$, and a.e.\ $(z,x,t)\in\R^n\X\Om\X(0,T)$, we have
\begin{equation}\label{e4.15}
 \a(z,x,t,u,\eta)\cdot \eta\ge c_\a |\eta|^2 +h_\a(x,t);
 \end{equation}
 
\medskip
 \item[{\bf($\mathbf{\a4}$)}] There exist $\tilde c_\a>0$, $\tilde h_\a\in L^2(\Om\X(0,T))$ such that for all $u\in\R$, $\eta\in\R^n$, and a.e.\ $(z,x,t)\in\R^n\X\Om\X(0,T)$, we have
  \begin{equation}\label{e4.15'}
  |\a(z,x,t,u,\eta)|\le c_\a (|u|+|\eta|) +h_\a(x,t);
  \end{equation}
 \medskip

 \item[{\bf($\mathbf{\a5}$)}] There exist $d_\a>0$, $0<\s<1$, such that, for all $v_1,v_2\in\R$, $\eta_1,\eta_2\in\R^n$, and for a.e.\  $(z,x,t)\in\R^n\X\Om\X(0,T)$, 
 \begin{equation}\label{e4.16}
  |\a(z,x,t,v_1,\eta_1)-\a(z,x,t,v_2,\eta_2)|\le d_\a(|v_1-v_2|^\s+|\eta_1-\eta_2|).
  \end{equation}
 \end{enumerate}

For $\psi$ satisfying {\bf($\mathbf{\psi1}$)}--{\bf($\mathbf{\psi4}$)}, let us define
\begin{equation}\label{e4.17}
\psi_0(x,t,v,\eta):= \inf_{\eta_1\in\Nu_{\pot}^2} \int_{\KK}\psi(z,x,t,v,\eta+\eta_1(z))\,d\mm(z), \quad  v\in\R,\ \eta\in\R^n, \text{a.e.\ $(x,t)\in\Om\X(0,T)$}.
\end{equation}

The following is a straightforward adaptation of lemma~5.2 of \cite{MR2349874} to extend the latter to the context of ergodic algebras.  Let $\psi,\a,\psi_0$ be as we just described.

\begin{lemma}\label{L:4.2} Assume that $\xi,\eta\in L^2(0,T;\BB^2(\R^n;L^2(\Om)))$ and that, for a.e.\ $t\in(0,T)$, $\xi(\cdot,\cdot,t), \eta(\cdot,\cdot,t)$ are regular test functions. 
Assume further, that $\xi(\cdot,x,t)\in\BB_{\pot}^2$ and $\eta(\cdot,x,t)\in\BB_{\sol}^2$, for a.e.\ $(x,t)\in\Om\X(0,T)$, and let $u\in L^2(\Om\X(0,T))$. If 
\begin{equation}\label{eL4.2}
\eta(z,x,t)=\a(z,x,t, u(x,t), \xi(z,x,t)),\quad\text{for a.e.\ $(z,x,t)\in\R^n\X\Om\X(0,T)$},
\end{equation}
then $\mv \eta(x,t)\in\partial\psi_0(x,t,u(x,t),\mv \xi(x,t))$ for   a.e.\ $(x,t)\in\Om\X(0,T)$.
\end{lemma}
\begin{proof}
Let $\eta_1\in{\mathcal V}^2_{{\pot}},$  and $ \lambda\in\R$. Using~\eqref{eL4.2}, and the orthogonality of the spaces $\BB^2_{{\pot}}$ and $\BB^2_{{\sol}}$, see~\eqref{e2.4},  we obtain
$$
\int_{\KK}\psi\left(z,x,t,u(x,t),\xi(z,x,t)\right)\,d\mm(z)-
\int_{\KK}\psi\left(z,x,t,u(x,t), \lambda+\eta_1(z)\right)\,d\mm(z)\le \mv{\eta}\,(\mv{\xi}(x,t)-\lambda).
$$
From the definition of $\psi_0$, we have that $\psi_0\left(x,t,u(x,t),\mv{\xi}(x,t)\right)\le\int_{\KK}\psi\left(z,x,t,u(x,t),\xi(z,x,t)\right)\,d\mm(z)$, and therefore
\begin{equation}\label{eL4.2-2}
\psi_0\left(x,t,u(x,t), \mv{\xi}(x,t)\right)-\mv{\eta}\,(\mv{\xi}(x,t)-\lambda)
\le \int_{\KK}\psi\left(z,x,t,u(x,t),\lambda+\eta_1(z)\right)\,d\mm(z).
\end{equation}
Taking the infimum  in \eqref{eL4.2-2} with respect to $\eta_1$, we obtain the desired result.
\end{proof}

We finally state and prove the theorem which is the second main tool for the homogenization analysis developed in the next section.

\begin{theorem} \label{T:4.2} Let $u_{\ve_i}$, $\ve_i\downarrow 0$, be a sequence uniformly bounded in  $L^2(0,T: H^1(\Om))$, strongly converging in $L^p(\Om\X(0,T))$ for any $1\le p<2$. By passing to a subsequence, if necessary, we may assume that $\nabla_x u_{\ve_i}-\nabla_x u$ two-scale converges to $u_1\in \BB^2(\R^n; L^2(\Om\X(0,T)))$. Let 
\begin{equation}\label{eT4.2-1}
q_\ve(x,t):=\a(\frac{x}{\ve},x,t,u_\ve(x,t),\nabla_x u_\ve(x,t)),
\end{equation}
where $\a=\nabla_\eta\psi$, with $\psi$ and $\a$ satisfying {\bf($\mathbf{\psi1}$)}--{\bf($\mathbf{\psi4}$)}, {\bf($\mathbf{\a1}$)}--{\bf($\mathbf{\a5}$)}. Again, passing to a further subsequence, if necessary, we may also assume that $q_{\ve_i}$ two-scale converges to 
$q\in\BB^2(\R^n;L^2(\Om\X(0,T)))$. Suppose we have
\begin{equation}\label{eT4.2-2}
\lim_{\ve\to 0}\int_0^T\bigg(\int\limits_{\Omega\X(0,t)}q_\ve\cdot\nabla u_\ve\,dx\,ds\bigg)\,dt=
\int_0^T\int\limits_{\Omega\X(0,t)}\int_{\KK}q(z,x,t)\cdot\big(\nabla u(x,s)+u_1(z,x,s)\big)\,d\mm(z)\,dx\,ds\,dt.
\end{equation}
Then, $q(z,x,t)=\a(z,x,t, u(x,t), \nabla_x u(x,t)+u_1(z,x,t))$, ($d\mm(z)\,dx\,dt$)-a.e.\ in $\KK\X\Om\X (0,T)$. In particular,  
\begin{equation}\label{eT4.2-3}
\mv q(x,t)\in\po\psi_0(x,t,u(x,t),\nabla_x u(x,t)),\quad\text{for a.e.\ $(x,t)\in\Om\X(0,T)$},
\end{equation}
where $\psi_0$ is defined by \eqref{e4.17}.
\end{theorem}

\begin{proof}  Given $\phi\in\AA\left(\R^n; C_0^\infty(\Omega\X(0,T);\R^n)\right)$ and using the monotonicity of $\alpha$, we get
\begin{equation}\label{eT4.2-6}
0\leq \int_0^T\bigg(\int\limits_{\Omega\X(0,t)}\bigg(q_\ve-\alpha\left(\frac{x}{\ve},x,s,u_\ve, \phi_\ve(x,s)\right)\bigg)\cdot \left(\nabla u_\ve-\phi_\ve\right)\,dx\,ds\bigg)\,dt,
\end{equation}
where $\phi_\ve(x,t):=\phi\left(\frac{x}{\ve},x,t\right)$.

However,
\begin{eqnarray*}
&&\int_0^T\bigg(\int\limits_{\Omega\X(0,t)}\bigg|\left(\alpha\left(\frac{x}{\ve},x,s,u_\ve,\phi_\ve(x,s)\right)-
\alpha\left(\frac{x}{\ve},x,s, u,\phi_\ve(x,s)\right)\right)\cdot
\left( \nabla u_\ve-\phi_\ve\right)\bigg|\,dx\,ds\bigg)\,dt\leq\\
&&\qquad\int_0^T\bigg(\int\limits_{\Omega\X(0,t)}\bigg|\alpha\left(\frac{x}{\ve},x,s, u_\ve,\phi_\ve(x,s)\right)-
\alpha\left(\frac{x}{\ve},x,s, u,\phi_\ve(x)\right)\bigg|^2\,dx\,ds\bigg)^{1/2} \\
&&\qquad\qquad\times \bigg(\int\limits_{\Omega\X(0,t)}\bigg(\left|\nabla u_\ve-\phi_\ve\right|^2\,dx\,ds\bigg)^{1/2}\,dt.
\end{eqnarray*}
Since $u_\ve$ is bounded in $L^2\left(0,T;H^1(\Omega)\right)$,  and \eqref{e4.16}  holds, from the   strong convergence of $u_\ve$ to $u$ in $L^p(\Omega\X(0,T))$, for $1\le p<2$, we obtain from the last inequality
\begin{eqnarray}\label{eT4.2-7}
&&\limsup_{\ve\to 0}\int_0^T\bigg(\int\limits_{\Omega\X(0,t)}\bigg|\left(\alpha\left(\frac{x}{\ve},x,s,u_\ve, \phi_\ve(x,s)\right)-
\alpha\left(\frac{x}{\ve},x,s,u,\phi_\ve(x,s)\right)\right)\cdot
\left( \nabla u_\ve-\phi_\ve\right)\bigg|\,dx\,ds\bigg)\,dt\leq\nonumber\\
&&\qquad\qquad\qquad C \limsup_{\ve \to 0}\int_0^T\bigg(\int\limits_{\Omega\X(0,t)}\left|u_\ve-u\right|^{2\sigma}\,dx\,ds\bigg)^{1/2}\,dt=0.
\end{eqnarray}
Then, by Lemma~\ref{L:4.1}, 
\begin{multline}
\lim_{\ve\to 0}\int_0^T\bigg(\int\limits_{\Omega\X(0,t)}
\alpha\left(\frac{x}{\ve},x,s,u_\ve, \phi_\ve(x,s)\right)\cdot \left( \nabla u_\ve-\phi_\ve\right)\,dx\,ds\bigg)\,dt=\\
\lim_{\ve\to 0}\int_0^T\bigg(\int\limits_{\Omega\X(0,t)}
\alpha\left(\frac{x}{\ve},x,s,u,\phi_\ve(x,s)\right)\cdot \left(\nabla u_\ve-\phi_\ve \right)\,dx\,ds\bigg)\,dt=\\
\int_0^T\bigg(\int\limits_{\Omega\X(0,t)\times\KK}
\alpha\left(z,x,s,u, \phi(z,x,s)\right)\cdot
\left(\nabla u(x,s)+u_1(z,x,s)-\phi(z,x,s))\right)\,d\mm(z)\,dx\,ds\bigg)\,dt.
\end{multline}
Set $v(z,x,t):= \nabla u(x,t)+u_1(z,x,t)$. 
From the last  equality, \eqref{eT4.2-2}, and  \eqref{eT4.2-6}, we obtain
\begin{equation}\label{eT4.2-8}
0\leq \int_0^T\bigg(\int\limits_{\Omega\X(0,t)\times\KK}\bigg(q-\alpha\left(z,x,s,u(x,s),\phi(z,x,s)\right)\bigg)\cdot \left(v(z,x,t)-\phi(z,x,s) \right)\,d\mm(z)\,dx\,ds\bigg)\,dt.
\end{equation}
Now, we take, in \eqref{eT4.2-8},  
$$
\phi(z,x,t)=v_\d(z,x,t)\pm\delta \varphi(x,t) \sigma(z) e_i, 
$$
 where  $\delta>0$, $(\sigma, \varphi)\in \AA\times C_0^{\infty}(\Omega\X(0,T))$, $e_i$ is the $i$-th element of the canonical basis in $\R^n$, 
and $v_\d$ is a sequence in $\AA(\R^n; C_0^\infty(\Om\X(0,T);\R^n))$ converging to $v$ in $\BB^2(\R^n;L^2(\Om\X(0,T);\R^n))$, as $\d\to0$. 
We get
$$
0\leq \int_0^T\bigg(\int\limits_{\Omega\X(0,t)\times\KK}\bigg(q-\alpha\left(z,x,s,u(x,s), -v_\d\mp\delta\varphi(x,t)\sigma(z)e_i\right)\bigg)\cdot \left(\pm \varphi(x,t)\sigma(z)e_i\right)\delta\,d\mm(z)\,dx\,ds\bigg)\,dt.
$$
Dividing the last inequality by $\delta$ and letting $\delta\to 0$, using again \eqref{e4.16},  yields
$$
\int_0^T\bigg(\int\limits_{\Omega\X(0,t)\times\KK}\big(q-\alpha\left(z,x,s,u(x,s), v(z,x,s)\right)\big)\cdot e_i\,\varphi(x,t)\,\sigma(z)\,d\mm(z)\,dx\,ds\bigg)\,dt=0,
$$
and therefore, we conclude that $q(z,x,t)=\alpha\left(z,x,t,u(x,t), v(z,x,t)\right)$.  Finally,  by the Lemma~\ref{L:4.2} we have that
$\overline{q}(x,t)\in\partial\psi_0\left(x,t, u(x,t), \nabla u(x,t)\right)$ for a.e.\ $(x,t)\in\Omega\X(0,T)$,  completing  the proof of theorem.
\end{proof}

We conclude this section by recalling  some facts about  the conjugate function $G^*$ of a convex function $G:\R^m\to\R$, which we will be used in the next section. 

We recall that $G^*$ is defined by
$$
G^*(v):=\sup \{ v\cdot\xi-G(\xi)\,:\, \xi\in\R^n\}.
$$
We may verify, from the definition, that $w\in\partial G(u)$ if, and only if,
\begin{equation}\label{e4.conj}
G(u)+G^*(w)=u\cdot w.
\end{equation}
For real valued convex functions we also have $G^{**}(u)=G(u)$,   and, so, $w\in\po G(u)$ if, and only if, $u\in\po G^*(w)$. 
We will also use the fact that if $G:\R^m\to\R$ is strictly convex, then $G^*$ is everywhere differentiable, and so $w\in\partial G(u)$ implies $u=\nabla G^*(w)$. 
For all these facts about convex functions we refer to, {\em e.g.}, \cite{MR0274683,MR2798533}.

In particular, if $w\in H^1(0,T; H^{-1}(\Om))\bigcap L^2(0,T; L^2(\Om))$,  $u\in L^2(0,T; H_0^1(\Om))$,  $G:\R\to\R$ is strictly convex, and $w(x,t) \in \partial G(u(x,t))$, for a.e.\ $(x,t)\in\Om\X(0,T)$, we have
\begin{equation}\label{e4.conj2}
\int_\tau^t \la w_t (s),u(s) \ra_{H^{-1}, H_0^1}\,dt = \int_{\Om}(G^*(w(t))-G^*(w(\tau)))\,dx,\quad\text{for a.e.\ $0<\tau<t<T$},
\end{equation}
which may be easily deduced, by approximation, from the trivial smooth case.  

 \section{Homogenization of the generalized Stefan problem.}\label{S:5}

 In this section we apply the framework developed in the preceding sections to analyze the homogenization of the following initial-boundary value problem 
 \begin{equation}\label{e5.1}
 \begin{aligned}
 &\po_t w_\ve -\nabla\cdot \a(\frac{x}{\ve},x,t, u_\ve,\nabla u_\ve)=f(\frac{x}{\ve}, x, u_\ve),\\
 & w_\ve(x,t)\in \partial \Psi(\frac{x}{\ve},x,u_\ve),\qquad \text{a.e.\ in $\Om\X(0,T)$},\\
 &w_\ve(x,0)=w_0(\frac{x}{\ve},x),\qquad x\in\Om,\\
 & u_\ve=0,\qquad \text{in $\po\Om\X(0,T)$}.
 \end{aligned}
\end{equation}
Here, $\Psi$ and  $\a$ satisfy {\bf($\mathbf{\Psi1}$)}--{\bf($\mathbf{\Psi4}$)}, $m=1$, {\bf($\mathbf{\psi1}$)}--{\bf($\mathbf{\psi4}$)}, and {\bf($\mathbf{\a1}$)}-{\bf($\mathbf{\a5}$)} in Section~\ref{S:4}. For $f$ we assume
\begin{enumerate}

\item[{\bf($\mathbf{f1}$)}] $f(z,x,t,\cdot)$ is continuous for a.e.\ $(z,x,t)\in\R^n\X\Om\X(0,T)$;
\medskip
\item[{\bf($\mathbf{f2}$)}] $t\mapsto f(\cdot,\cdot,t , u)\in L^2(0,T; \BB^2(\R^n;L^2(\Om)))$ and  $f(\cdot,\cdot,t, u)$ is a regular test function for a.e.\ $t\in(0,T)$; 
\medskip
\item[{\bf($\mathbf{f3}$)}] There exists $c_f>0$, $h_f\in L^s(\Om\X(0,T))$, with $2<s\le\infty$, and $0<\s<1$ such that for all $u\in\R$ and ,
\begin{equation}\label{e5.2}
|f(z,x,t,u)|\le c_f |u|^\s+h_f(x,t),\qquad \text{for a.e.\ $(z,x,t)\in\R^n\X\Om\X(0,T)$}.
\end{equation}

\end{enumerate}

Let us also assume that  $w_0\in \BB^2(\R^n;L^2(\Om))\cap L_\loc^2(\R^n\X\Om)$.

\begin{definition}\label{D:5.1}   We say that the pair $(u_\ve(x,t),w_\ve(x,t))$ is a weak solution of \eqref{e5.1}, in $\Om\X(0,T)$, if 
\begin{align}
& u_\ve\in L^2(0,T;H_0^1(\Om)),\  w_\ve\in H^1(0,T; H^{-1}(\Om))\cap L^2(\Om\X(0,T)),\label{e5.3}\\
& w_\ve\in\po\Psi (\frac{x}{\ve}, x,u_\ve) \qquad \text{for a.e.\ $(x,t)\in \Om\X(0,T)$}, \label{e5.4}
\end{align}
and for all $v\in H^1(0,T; L^2(\Om))\cap L^2(0,T; H_0^1(\Om))$, with $v(\cdot,T)=0$, a.e.\ in $\Om$, we have
\begin{equation}\label{e5.5}
\int\limits_{\Om\X(0,T)} \left((w_\ve-w_0) v_t-q_\ve\cdot \nabla v+f_\ve v\right)\, dx\,dt=0,
\end{equation}
where we set
\begin{align*}
& q_\ve(x,t)= \a(\frac{x}{\ve},x,t,u_\ve,\nabla u_\ve),\\
& F_\ve(x,t)=f(\frac{x}{\ve},x,t,u_\ve).
\end{align*}
\end{definition}

Existence of a weak solution for the problem \eqref{e5.1}, in the sense of the Definition~\ref{D:5.1}, has been proved by Visintin in  \cite{MR2349874}. Here we will be only concerned with the homogenization problem related with~\eqref{e5.1}.
   
We  now  state and prove the main result of this paper.  Let $\Psi$ and  $\a$ satisfy {\bf($\mathbf{\Psi1}$)}--{\bf($\mathbf{\Psi4}$)}, $m=1$, {\bf($\mathbf{\psi1}$)}--{\bf($\mathbf{\psi4}$)}, and {\bf($\mathbf{\a1}$)}-{\bf($\mathbf{\a5}$)} in Section~\ref{S:4}, and $f$ satisfy {\bf($\mathbf{f1}$)}-{\bf($\mathbf{f3}$)}.

\begin{theorem}\label{T:5.1}
Let $\{(u_{\ve},w_\ve) \}_{\ve>0}$ be a family of weak solutions of problem~\eqref{e5.1}. Then,  there exists a subsequence  $(u_{\ve_i}, w_{\ve_i})$, such that $u_{\ve_i}$  strongly  converges in $L^p(\Om\X(0,T))$, for all $1\le p<2$, to  a function $u\in L^2(\Omega\X(0,T))\cap L^2\left(0,T;H_0^1(\Omega)\right)$, $w_{\ve_i}$ weakly converges in 
$L^2(\Om\X(0,T))$ to  $w\in L^2(\Om\X(0,T))\cap H^1(0,T; H^{-1}(\Om))$, and the pair $(u,w)$ is a weak solution, in $\Om\X(0,T)$, of the  problem
\begin{equation}\label{e5.6}
\begin{aligned}
& w_t-\nabla\cdot \mv{q}= \mv{f}\left(x,t,u\right), \\
& w(x,t)\in \partial \mv{\Psi}\left(x,u(x,t)\right),\\
&\mv{q}\in\partial{\psi}_0\left (x,t,u(x,t), \nabla u(x,t)\right),\quad \text{for a.e. $(x,t)\in \Omega\X(0,T)$},\\
&w(x,0)=\mv w_0(x), \quad x\in \Omega,\quad u=0,\quad\text{in $\partial\Omega\X(0,T)$}.
\end{aligned}
\end{equation}
Here,  $\mv{\Psi}(x,\lambda), \ \mv f, \mv w_0$ are the mean values of $\Psi(\cdot,x,\lambda), f(\cdot,x,t,\l), w_0(\cdot,x)$ for each $(x,t,\lambda) \in  \Omega\X(0,T)\X\R$, and $\partial \mv{\Psi}$ represents the subdifferential of the function $\mv{\Psi}(x, \cdot)$ for each $x\in \Omega$. Also,  $\psi_0$ is defined by \eqref{e4.17}, and we mean by a weak solution of \eqref{e5.6} a pair  $(u,w)$ satisfying \eqref{e5.3}, \eqref{e5.4}, \eqref{e5.5}, with $u_\ve,w_\ve, \Psi, q_\ve,f_\ve,w_0$, replaced by $u,w,  \mv \Psi, \mv q,\mv f, \mv w_0$, with $\mv q$ satisfying the third relation in \eqref{e5.6}.
\end{theorem}

\begin{proof} By \eqref{ePsi3}, it follows that $|\Psi(z,x,\xi)|\le  c|\xi|^2+ |h(z,x)||\xi |+ |\Psi(z,x,0)|$, and from the definition of the conjugate $\Psi^*$ we have $\xi v\le \Psi^*(z,x,v)+\Psi(z,x,\xi)$, for any $\xi, v\in\R$. Choosing $\xi=v/(2c)$, we get $\Psi^*(z,x,v)\ge \frac{v^2}{2c}-|h(z,x)||v|-|\Psi(z,x,0)|\ge \frac{v^2}{4c}-c|h(z,x)|^2-|\Psi(z,x,0)|$. In particular, there exist constants 
$\gamma>0$, $\tilde \gamma\in\R$, such that,   we have
\begin{equation}\label{e5.7}
\int_{\Om} \Psi^*(\frac{x}{\ve},x, v(x))\,dx\ge \gamma \|v\|_{L^2(\Om)}^2+\tilde \gamma, \quad \text{for any function   $v\in L^2(\Om\X(0,T))$ and $\ve>0$}. 
\end{equation}
{}From \eqref{e5.5}, it follows
\begin{equation}\label{e5.8}
 \int_0^t \la \po_tw_\ve(s), v(s) \ra_{H^{-1},H_0^1}\,ds +\int\limits_{\Om\X(0,t)} \left(q_\ve\cdot \nabla v-f_\ve v\right)\, dx\,dt=0,
\end{equation}
for all $v\in H^1(0,T; L^2(\Om))\cap L^2(0,T; H_0^1(\Om))$, with $v(\cdot,s)=0$, $t\le s\le T$, a.e.\ in $\Om$. Clearly, \eqref{e5.8} extends to any  $v\in L^2(0,t; H_0^1(\Om))$, and so we may take $v(x,s)=u_\ve(x,s)$, $(x,s)\in\Om\X(0,t)$,  to obtain 
\begin{equation}\label{e5.9}
 \int_0^t \la \po_tw_\ve(s), u_\ve(s) \ra_{H^{-1},H_0^1}\,ds+\int\limits_{\Om\X(0,t)} \left(q_\ve\cdot \nabla u_\ve-f_\ve u_\ve\right)\, dx\,ds=0,\quad\text{for a.e.\ $t\in(0,T)$}.
\end{equation} 
Using \eqref{e4.conj2}, \eqref{e5.7}, \eqref{e4.15} and \eqref{e5.2}, we obtain 
\begin{equation}\label{e5.10}
\|w_\ve(t)\|_{L^2(\Om)}^2 +\int_0^t\int_{\Om}|\nabla u_\ve|^2\, dx\,ds\le C \left(1+ \int_0^t\int_\Om |u_\ve(x,s)|^{1+\s}\,dx\,ds\right),
\end{equation}
which, by the trivial inequality $|u|^{1+\s}\le \frac{1+\s}2\g|u|^2+\frac{1-\s}2 C(\g)$, together with  Poincar\'e inequality, choosing $\g>0$ sufficiently small, 
 implies the uniform bounds
\begin{equation}\label{e5.11}
\|w_\ve\|_{L^\infty(0,T; L^2(\Om))}\le c, \quad \|u_\ve\|_{L^2(0,T;H_0^1(\Om))}\le c,\quad\text{for some $c>0$ independent of $\ve>0$.}
\end{equation}
Therefore, from \eqref{e4.15} and \eqref{e5.2} we have
\begin{equation}\label{e5.12}
\|q_\ve\|_{L^2(\Om\X(0,T))}\le c,\quad \|f_\ve\|_{L^2(\Om\X(0,T))}\le c,\quad\text{for some $c>0$ independent of $\ve>0$},
\end{equation}
and from \eqref{e5.8} we also obtain
\begin{equation}\label{e5.12'}
\|w_\ve\|_{H^1(0,T;H^{-1}(\Om))}\le c, \quad  \text{for some $c>0$ independent of $\ve>0$.}
\end{equation}

We may then apply Theorem~\ref{T:4.1} to conclude that there exist a subsequence $\ve_i\downarrow0$ and functions $u\in L^2(0,T;H^1(\Om))$, $w\in L^2(\Om\X(0,T))\cap H^1(0,T;H^{-1}(\Om))$,  such that 
\begin{equation}\label{e5.13}
\begin{cases}
\text{$u_{\ve_i}$ strongly converges to $u$ in $L^p(\Om\X(0,T))$, for any $1\le p<2$}, \\
\text{ $w_{\ve_i}$ weakly converges to $w$ in $L^2(\Om\X(0,T))$, and}\\
\text{$w(x,t)\in \po\mv\Psi(x,t,u(x,t))$, a.e.\ in $\Om\X(0,T)$.} 
\end{cases}
\end{equation}

By passing to a further subsequence, if necessary, we may  assume that $q_{\ve_i}$ and $F_{\ve_i}$ two-scale converge to 
$$
q(z,x,t), F(z,x,t)\in\BB^2(\R^n; L^2(\Om\X(0,T))).
$$
Taking, in \eqref{e5.5},  $v=\ve\phi(x)\z(x/\ve)$, with $\phi\in C_0^\infty(\Om)$ and $\z\in\AA$, such that $\nabla_z\z\in\AA$, and letting $\ve\to0$,  we get
\begin{equation}\label{e5.13-0}
\int\limits_{\Om\X(0,T)}\int_\KK q(z,x,t)\cdot \nabla\z(z)\phi(x)\,d\mm(z)\,dx\,dt=0,
\end{equation}
and, so, we deduce that
\begin{equation}\label{e5.13-1}
q(\cdot,x,t)\in\Nu_{\sol}^2\quad\text{for a.e.\ $(x,t)\in\Om\X(0,T)$}.
\end{equation}
 
By the proof of Theorem~\ref{T:4.1}, we see that, given any subnet of $u_{\ve_i}$ generating a two-scale Young measure $\nu_{z,x,t}$, we must have $\nu_{z,x,t}=\d_{u(x,t)}$, for
$(d\mm(z)\,dx\,dt)$-a.e.\ $(z,x,t)\in\KK\X\Om\X(0,T)$. Therefore, from Theorem~\ref{T:3.2}, we deduce
$$
F(z,x,t)=f(z,x,t,u(x,t)),\quad \text{for $(d\mm(z)\,dx\,dt)$-a.e.\ $(z,x,t)\in\KK\X\Om\X(0,T)$}.
$$
In particular, $f_\ve\wto \mv f=\int_{\KK} f(z,\cdot,\cdot,u(\cdot,\cdot))\,d\mm(z)$, in $L^2(\Om\X(0,T))$.  But, from \eqref{e5.2}, we see that $f_{\ve_i}$ is uniformly bounded in $L^s(\Om\X(0,T))$, for some $s>2$, and so, by passing to a further subsequence, if necessary, still denoted $f_{\ve_i}$, we deduce that 
\begin{equation}\label{e5.13'}
f_{\ve_i}(\cdot,\cdot) \wto \mv f(\cdot,\cdot, u(\cdot,\cdot))\quad \text{in $L^s(\Om\X(0,T))$, for some $s>2$}.
\end{equation}

We now prove that \eqref{eT4.2-2} holds, in order to finish the proof, by applying Theorem~\ref{T:4.2}. We will drop the subscript from $\ve_i$ to simplify the notation.
 
{}From   $w_\epsilon \in \partial \Psi (x/\epsilon,x, u_\ve)$,  we have
$$
\int\limits_{\Omega\X(\tau,t)}\left( \Psi(x/\ve,x, u_\ve(x,s)) + \Psi^*(x/\ve,x,w_\ve(x,s)) \right) \,dx\,ds = \int\limits_{\Omega\X(\tau,t)}  u_\ve(x,s) w_\ve(x,s) \,dx\,ds, \quad\text{for a.e.\ $0<\tau<t<T$}.
$$
Taking the limit $\ve \rightarrow 0$ in the last equation, and applying \eqref{e5.13} we have
\begin{equation*}
\lim_{\ve \rightarrow  0} \int\limits_{\Omega\X(\tau,t)}    \Psi^*(x/\ve,x,w_\ve(x,s))\,dx\,ds  =    \int\limits_{\Omega\X(\tau,t)}  u(x,t) w(x,t)\,dx\,dt
  -  \int\limits_{\Omega\X(\tau,t)} \overline{\Psi} (x,u(x,s)) \,dx\,ds,
\end{equation*}
and therefore, again by \eqref{e5.13}, we conclude that
\begin{equation}\label{e5.14}
\lim_{\ve \rightarrow  0}\int\limits_{\Omega\X(\tau,t)} \Psi^*(x/\ve,x,w_\ve(x,s)) \,dx\,ds = \int_{\Omega\X(\tau,t)}  \overline{\Psi}^* (x, w(x,s))    \,dx\,ds.
\end{equation}
{}Further, from \eqref{e5.13}, we also have
\begin{equation}\label{e5.15}
\int_\tau^t\langle\partial_t w,u\rangle\,ds=\int_{\Omega}\big(
\overline{\Psi}^*\left(x, w(x,t)\right)-\overline{\Psi}^*\left(x,w(x,\tau)\right)\big)\,dx,
~~\text{for a.e. $\tau, t\in(0,T)$}.
\end{equation}
Also, taking the limit as $\ve\to0$ in \eqref{e5.8}, we get
\begin{equation}\label{e5.16}
 \int_0^t \la \po_tw(s), v(s) \ra_{H^{-1},H_0^1}\,ds +\int\limits_{\Om\X(0,t)} \left(\mv{q}\cdot \nabla v-\mv{f} v\right)\, dx\,dt=0,
\end{equation}
for all $v\in H^1(0,T; L^2(\Om))\cap L^2(0,T; H_0^1(\Om))$, with $v(\cdot,s)=0$, $t\le s\le T$, a.e.\ in $\Om$, which, again, can be extended to $v\in L^2(0,T;H_0^1(\Om))$, to give
\begin{equation}\label{e5.16'}
 \int_\tau^t \la \po_tw(s), v(s) \ra_{H^{-1},H_0^1}\,ds +\int\limits_{\Om\X(\tau,t)} \left(\mv{q}\cdot \nabla v-\mv{f} v\right)\, dx\,dt=0,\quad\text{for all $v\in L^2(0,T;H_0^1(\Om))$}.
\end{equation}
Since $u\in L^2\left([0,T];H^1_0(\Omega)\right)$, we take $u$  as a test function in \eqref{e5.16} to obtain
\begin{equation}\label{e5.17}
\int_{\Omega}\overline{\Psi}^*\left(x,w(x,t)\right)\,dx -L_0=
-\int\limits_{\Omega\X(0,t)}\bigg\{\overline{q}\cdot\nabla u -u\,\overline{f}\,\bigg\}\,dx\,ds, \quad \text{for a.e.\ $t\in[0,T]$},
\end{equation}
where, by $L_0$, we denote the following limit 
\begin{equation}\label{e5.17'}
L_0:=\lim_{t\to0}\int_{\Om} \overline{\Psi}^*\left(x, w(x,t)\right)\,dx,
\end{equation}
whose existence is guaranteed by the continuity in $\tau=0$ of the  right-hand side of  \eqref{e5.16'}.

We can rewrite \eqref{e5.9} as
\begin{equation}\label{e5.9'}
 \int_\tau^t \la \po_tw_\ve(s), u_\ve(s) \ra_{H^{-1},H_0^1}\,ds+\int\limits_{\Om\X(\tau,t)} \left(q_\ve\cdot \nabla u_\ve-f_\ve u_\ve\right)\, dx\,ds=0,\quad\text{for a.e.\ $0<\tau<t<T$},
\end{equation} 
which gives, by \eqref{e4.conj2}, 
\begin{multline}\label{e5.9''}
 \int\limits_{\Om}\left(\Psi^*(\frac{x}{\ve},x,u_\ve(x,t))-\Psi^*(\frac{x}{\ve},x,u_\ve(x,\tau)\right) \,dx=-\int\limits_{\Om\X(\tau,t)} \left(q_\ve\cdot \nabla u_\ve-f_\ve u_\ve\right)\, dx\,ds, \\
  \text{for a.e.\ $0<\tau<t<T$},
\end{multline} 
We then integrate \eqref{e5.9''} in $t$, from $\tau$ to $T$, and in $\tau$, from $0$ to $h$,  
\begin{eqnarray}\label{e5.18}
&&-\frac1h\int_0^h\,d\tau\int_\tau^T\bigg(\int\limits_{\Omega\X(\tau,t)}q_\ve\cdot\nabla u_\ve\,dx\,ds\bigg)\,dt=
\frac1h\int_0^h\,d\tau\int\limits_{\Omega\X(\tau,T)}\bigg(\Psi^*\left(\frac{x}{\ve},x,w_\ve(x,t)\right)
-\Psi^*\left(\frac{x}{\ve},x, w(\frac{x}{\ve},x,\tau)\right)\bigg)\,dx\,dt\nonumber\\
&&\qquad\qquad\qquad\qquad-\frac1h\int_0^h\,d\tau\int_\tau^T\int_{\Omega\X(\tau,t)}u_\ve\,f_\ve\,\,dx\,ds\,dt.
\end{eqnarray}
Taking $\ve\to0$ in \eqref{e5.18}, using \eqref{e5.14}, \eqref{e5.13'}, and then taking $h\to0$, we get
\begin{equation}\label{e5.19}
-\lim_{\ve\to 0}\int_0^T\bigg(\int\limits_{\Omega\X(0,t)}q_\ve\cdot\nabla u_\ve\,dx\,ds\bigg)\,dt=
\int\limits_{\Omega\X(0,T)}\overline{\Psi}^*\left(x, w(x,t)\right)\,dx\,dt - L_0
-\int_0^T\bigg(\int\limits_{\Omega\X(0,t)}u\,\overline{f}\,dx\,ds\bigg)\,dt
\end{equation}
where we have used  that 
\begin{equation}\label{e5.20}
\lim_{h\to0}\frac1h\int_0^h\sup_{\ve>0}\left(\int_0^\tau\int\limits_{\Omega\X(t,\tau)}|q_\ve\cdot\nabla u_\ve|\,dx\,ds\,dt+\int_0^T\int\limits_{\Omega\X(0,\tau)}|q_\ve\cdot\nabla u_\ve|\,dx\,ds\,dt\right)\,d\tau =0,
\end{equation}
 by \eqref{e5.11} and \eqref{e5.12}.
Now, using  \eqref{e5.17} in the right-hand side of  \eqref{e5.19}, we  get
\begin{equation}\label{e5.21}
\lim_{\ve\to 0}\int_0^T\bigg(\int\limits_{\Omega\X(0,t)}q_\ve\cdot\nabla u_\ve\,dx\,ds\bigg)\,dt=\int_0^T\int\limits_{\Omega\X(0,t)}\overline{q}\cdot\nabla u \,dx\,ds\,dt.
\end{equation}
Finally, applying Lemma~\ref{L:4.1}, \eqref{e5.13-1} and \eqref{e2.4},  we conclude
\begin{equation}\label{result5}
\lim_{\ve\to 0}\int_0^T\bigg(\int\limits_{\Omega\X(0,t)}q_\ve\cdot\nabla u_\ve\,dx\,ds\bigg)\,dt=
\int_0^T\int\limits_{\Omega\X(0,t)}\int_{\KK}q(z,x,t)\cdot\big(\nabla u+u_1(z,x,t)\big)\,d\mm(z)\,dx\,ds\,dt.
\end{equation}
We can then apply Theorem~\ref{T:4.2} to conclude the proof.

\end{proof}

\section{Single nonlinearity, Kirchhoff transformation, and uniqueness.}\label{S:6}

Following  \cite{MR2349874}, in this section we briefly comment on the simpler case where $\a(z,x,t,u,\eta)$ is linear in the vector variable $\eta$, that is,
\begin{equation}\label{e6.1}
\a(z,x,t,u,\eta)=K(z,x,t,u)\cdot\eta, \quad\text{for all $(u,\eta)\in\R\X\R^n$, and a.e.\ $(z,x,t)\in\R^n\X\Om\X(0,T)$},
\end{equation}
 where $K$ is a (possibly non-symmetric) positive-definite $n\X n$ matrix-valued function, so that there exist $k_0>0,\  k_1>0$ such that
 \begin{equation}\label{e6.1'}
k_0|\xi|^2\le  \left[K(z,x,t,u)\cdot\xi\right]\cdot\xi\le k_1|\xi|^2,\quad\text{for all $\xi\in\R^n$, $u\in\R$, and a.e.\ $(z,x,t)\in\R^n\X\Om\X(0,T)$},
 \end{equation} 
 with entries $k_{ij}$ such that   $k_{ij}(z,x,t,\cdot)$ is continuous in $\R$, for a.e.\ $(z,x,t)\in\R^n\X\Om\X(0,T)$,  $t\mapsto k_{ij}(\cdot,\cdot,t,u)\in L^2(0,T;\BB^2(\R^n;L^2(\Om)))$,  and $k_{ij}(\cdot,\cdot,t,u)$ is a regular test function, for a.e.\ $t\in(0,T)$,  for all $u\in\R$.   
 Besides, consistently with {\bf($\mathbf{\a5}$)}, we ask that there exists $c_k>0$ and $\s\in(0,1)$, such that
 \begin{equation}\label{e6.2}
 |k_{ij}(z,x,t,u_1)-k_{ij}(z,x,t,u_2)|\le c_k|u_1-u_2|^\s, \quad\text{for all $u_,u_2\in\R$, and a.e.\ $(z,x,t)\in\R^n\X\Om\X(0,T)$}.
  \end{equation}
 In this case, the homogenization of the factor $\nabla\cdot\a(z,x,t,u,-\nabla u)$ in \eqref{e5.1} may be performed by using the classical approach for linear elliptic and parabolic equations (see, e.g.,   \cite{MR1987520}). Following the notation in  the proof of Theorem~\ref{T:4.2}, if $q(z,x,t)$ is the two-scale limit of 
 $q_\ve(x,t)=\a(x/\ve,x,t,u_\ve,\nabla u_\ve)$,   by the proof of that theorem we obtain
 \begin{align}
 &q(z,x,t)= K(z,x,t,u(x,t))\cdot(\nabla_x u(x,t)+\nabla_z V(z,x,t)),\label{e6.3} \\ 
 &\qquad \text{for ($\,d\mm(z)\,dx\,dt$)-a.e.\ $(z,x,t)\in\KK\X\Om\X(0,T)$},\nonumber\\
 &\int_{\KK}\left[K(z,x,t,u(x,t))\cdot(\nabla_x u(x,t)+\nabla_z V(z,x,t))\right]\cdot\nabla_z\varphi(z) \,d\mm(z)=0,\label{e6.4}\\
 &\qquad\text{for a.e.\ $(x,t)\in\Om\X(0,T)$,  for all $\varphi\in\AA$, such that $\partial_k \varphi\in\AA$, $k=1,\cdots,n$},\nonumber
 \end{align}
 where we now write $u_1(z,x,t)=\nabla_z V(z,x,t)$, for some $V\in \BB^2(\R^n;L^2(\Om\X(0,T))$,  which we may do, since $u_1(\cdot,x,t)\in\Nu_{\pot}^2$, for a.e.\ $(x,t)\in\Om\X(0,T)$.  

More precisely, using the identification of $\BB^2$ with $L^2(\KK;d\mm(z))$, we denote by $H_0^1(\KK)$ the completion of the space of 
$$
\AA_1:= \{\varphi\in\AA\,:\,  \partial_k\varphi\in\AA,\  k=1,\cdots,n,\ \mv\varphi=0\}, 
$$
with respect to the metric given by the inner product defined by
$$
\la \varphi, \tilde \varphi\ra_1:=\int_{\KK} \nabla \varphi (z)\cdot  \nabla\tilde\varphi(z)\,d\mm(z).
$$
 Let $H^{-1}(\KK)$ denote the dual of $H_0^1(\KK)$, and let $H(\KK):= \{\varphi\in L^2(\KK)\,:\,\mv \varphi=0\}$, and $H(\KK;\R^n)=H(\KK)
 \overset{\text{$n$ times}}{\X\cdots\X} H(\KK)$. For $\vec \varphi\in H(\KK;\R^n)$, we may define $\nabla\cdot\vec\varphi\in H^{-1}(\KK)$ by
 $$
 \la \nabla\cdot\vec\varphi,\psi\ra:=\la \vec\varphi,\nabla\psi\ra_0:=\int_{\KK}\vec\varphi(z)\cdot\nabla \psi(z)\,d\mm(z),\quad\text{for all $\psi\in H_0^1(\KK)$}.
 $$
 Also, by Riesz representation, given any $\xi\in H^{-1}(\KK)$, we may find $v\in H_0^1(\KK)$, such that
 $$
 \la v, \varphi\ra_1=\la \xi,\varphi\ra,\quad\text{ for all $\varphi\in H_0^1(\KK)$},
 $$
 and we denote $v:=\Delta^{-1}\xi$. In this way, we can define 
 $$
 V(z,x,t):=\Delta^{-1} (\nabla\cdot u_1).
 $$ 
 The homogenized operator associated with the factor $\nabla\cdot\a(x/\ve,x,t,u_\ve,-\nabla u_\ve)$ in \eqref{e5.1}, when $\a$ is given by \eqref{e6.1}, is then obtained from 
 \eqref{e6.4}, as long as we obtain a representation for $\nabla_zV(z,x,t)$ in terms of $\nabla_xu$, where $u$ is the weak limit of $u_\ve$ in $L^2(0,T;H_0^1(\Om))$. As usual, this is achieved  by writing 
 $$
 V(z,x,t)= W_1(z,x,t)\frac{\partial u}{\partial x_1}+\cdots +W_n(z,x,t)\frac{\partial u}{\partial x_n},
 $$
 where $W_i(z,x,t)\in H_0^1(\KK)$, for a.e.\ $(x,t)\in\Om\X(0,T)$, is the solution of 
 \begin{equation}\label{e6.5}
\int_{\KK}\left[K(z,x,t,u(x,t))\cdot\nabla_z W_i(z,x,t))\right]\cdot\nabla_z\varphi(z) \,d\mm(z)=-\int_{\KK}\left[K(z,x,t,u(x,t))\cdot e_i\right]\cdot\nabla_z\varphi(z) \,d\mm(z), 
\end{equation}
for all $\varphi\in H_0^1(\KK)$, $i=1,\cdots,n$, where $e_i$ is the $i$-th element of the canonical basis.  The existence and uniqueness of $W_i(\cdot,x,t)$ is guaranteed by Lax-Milgram theorem (see, e.g., \cite{MR2597943}). Therefore, we obtain the homogenized operator
\begin{equation}\label{e6.6}
\mv q(x,t)= K_0(x,t,u(x,t))\cdot\nabla_x u,
\end{equation}
where $K_0(x,t,u)$ is the matrix-valued positive-definite function whose entries are given by
\begin{equation}\label{e6.7}
\begin{aligned}
k_{0\,ij}(x,t,u)&:= \int_{\KK} e_i\cdot\left[ K(z,x,t,u)\cdot\left(e_j+\nabla_zW_j(z,x,t)\right)\right]\,d\mm(z)  \\
&=\int_{\KK} \left(e_i+\nabla_z W_i(z,x,t)\right)\cdot\left[K(z,x,t,u)\cdot\left(e_j+\nabla_zW_j(z,x,t)\right)\right]\,d\mm(z)\\
&=\int_{\KK} \left(\d_{il}+\partial_l W_i(z,x,t)\right)k_{lm }(z,x,t,u)\left(\d_{jm}+\partial_m W_j(z,x,t)\right)\,d\mm(z),
\end{aligned}
\end{equation}
where we use the summation convention. Observe that $K_0$ also satisfies an ellipticity condition like \eqref{e6.1'}. 

Let us now consider the particular case where
\begin{equation}\label{e6.8}
K(z,x,t,u)=G(z,x) h(u), \quad\text{so that}\quad k_{ij}(z,x,t,u)=g_{ij}(z,x) h(u),
\end{equation}
with $G=(g_{ij})$ a positive-definite matrix-valued function, say, for some $0<\gamma_0<\gamma_1<\infty$,
\begin{equation}\label{e6.8'}
\gamma_0|\xi|^2\le [G(z,x)\cdot\xi]\cdot\xi\le \gamma_1|\xi|^2,\quad\text{for all $\xi\in\R^n$, a.e.\ $(z,x)\in\R^n\X\Om$},
\end{equation}
 and $h:\R\to\R$ is a Borelian function, with $h(u)>0$, for a.e.\ $u\in\R$. It has a special interest since it includes both the classical Stefan model (see, e.g.,  \cite{MR613313})  and the classical porous medium equation (see, e.g., \cite{MR877986}), and the problem~\eqref{e5.1}, in this case, is well posed, as we will show now, considering just the situation where $f\equiv0$, for simplicity.

In this case, we may use the so called Kirchhoff transformation, defined as
$$
V:=H(u)=\int_0^uh(\s)\,d\s.
$$
Through this change of dependent variable we have
$$
q(z,x,t)=G(z,x)\cdot \nabla_xV.
$$
 The first equation in \eqref{e5.1}  then reads
 \begin{equation}\label{e6.9}
 \frac{\partial w}{\partial t}-\nabla\cdot \left[G(\frac{x}{\ve},x)\cdot\nabla_x V\right]=0
 \end{equation}
 while the second relation in \eqref{e5.1} becomes
 \begin{equation}\label{e6.10}
 w(x,t)\in \partial \tilde\Psi(\frac{x}{\ve},x,V(x,t)),\quad\text{a.e.\ in $\Om\X(0,T)$},
 \end{equation}
 and $\tilde \Psi(z,x,V)$ is the convex function (defined up to a constant) whose subdifferential is $\partial \Psi(z,x,H^{-1}(V))$, which is strictly convex whenever $\Psi$ is.  These relations are then complemented by the corresponding initial-boundary conditions 
\begin{equation}\label{e6.11}
w(x,0)= w_0(\frac{x}{\ve},x),\qquad V(x,t)=0,\ \text{for $(x,t)\in\po\Om\X(0,T)$}.
\end{equation} 
 The homogenized problem is then obtained through the argument just explained leading to \eqref{e6.7}. We then get a homogenized equation of the form
  \begin{equation}\label{e6.12}
 \frac{\partial w}{\partial t}-\nabla\cdot \left[G_0(x)\cdot\nabla_x V\right]=0
 \end{equation}
 with
 \begin{equation}\label{e6.13}
 w(x,t)\in \partial \mv{\tilde\Psi}(x,V(x,t)), \quad\text{a.e.\ in $\Om\X(0,T)$},
 \end{equation}
and initial-boundary conditions 
\begin{equation}\label{e6.14}
w(x,0)=\bar w_0(x),\qquad V(x,t)=0,\ \text{for $(x,t)\in\po\Om\X(0,T)$}.
\end{equation}
As we have anticipated above, an important feature about both problems, \eqref{e6.9}--\eqref{e6.11}, and \eqref{e6.12}--\eqref{e6.14}, is that they are well posed, that is, the weak solution in the sense of Definition~\ref{D:5.1} is unique.  This is a well known fact, but we  give here  a simple proof of that, for the sake of completeness of our discussion, and also
because we miss a precise reference for this specific case. We do that for problem \eqref{e6.12}--\eqref{e6.14}, which, of course, also applies to problem \eqref{e6.9}--\eqref{e6.11}. 
For simplicity, we will assume that $G_0(x)$ is symmetric.
 
Let us consider the bilinear form $a:H_0^1(\Om)\X H_0^1(\Om)\to\R$, given by
$$
a(u,v):=\la G_0(x)\cdot\nabla u,\nabla v\ra_{L^2(\Om)},
$$
and let $A: H_0^1(\Om)\to H^{-1}(\Om)$ be the corresponding operator, i.e., 
$$
Au(v):= a(u,v).
$$
 {}From \eqref{e6.8'} we easily deduce that $A$ is an isomorphism. Now, suppose $(w_1,V_1)$ and $(w_2,V_2)$ are two weak solutions of \eqref{e6.12}--\eqref{e6.14},
 and set $w=w_1-w_2$, $V=V_1-V_2$. By the Definition~\ref{D:5.1} we obtain that
 \begin{equation}\label{e6.15}
 \int_0^t\la \frac{\partial w}{\partial t}(s),v(s)\ra_{H^{-1},H_0^1}\,ds+\int_0^t\int_{\Om} [G_0(x)\cdot\nabla V(x,s)]\cdot\nabla v(x,s)\,dx\,ds=0,
 \end{equation}   
 for all $v\in L^2(0,T;H_0^1(\Om))$, or, equivalently,
 \begin{equation}\label{e6.16}
 \frac{\po w}{\po t} =- A V,\quad\text{in $L^2(0,T; H^{-1}(\Om))$}.
 \end{equation}
By applying $A^{-1}$ to both sides, we get
 \begin{equation}\label{e6.16'}
A^{-1} \frac{\po w}{\po t} =-  V,\quad\text{in $L^2(0,T; H_0^1(\Om))$}.
 \end{equation} 
 Therefore, we have
 $$
\int_0^t a\left(A^{-1} \frac{\po w}{\po t}(s), A^{-1}w(s)\right)\,ds=-\int_0^t a(V(s),A^{-1}w(s))\,ds,
$$
and so, since $a$ is symmetric,
 \begin{equation}\label{e6.17}
\int_0^t \frac{\po}{\po t}\left(\int_{\Om}[G(x)\cdot\nabla A^{-1}w(x,s)]\cdot( \nabla A^{-1}w(x,s))\,dx\right)\,ds=-\int_0^t \la V(s), w(s)\ra_{L^2(\Om)}\,ds.
\end{equation}
Now, $\nabla A^{-1}w\in C([0,T); L^2(\Om))$ and $\nabla A^{-1}w(0)=0$, so \eqref{e6.17} implies
 \begin{equation}\label{e6.18}
\int_{\Om}[G(x)\cdot\nabla A^{-1}w(x,t)]\cdot( \nabla A^{-1}w(x,t))\,dx=-\int_0^t \la V(s), w(s)\ra_{L^2(\Om)}\,ds,\quad\text{for a.e.\ $t\in(0,T)$}.
\end{equation}
But $w(x,s)=w_1(x,s)-w_2(x,s)\in \po\mv{\tilde\Psi}(x,V_1(x,s))-\po\mv{\tilde\Psi}(x,V_2(x,s))$, for a.e.\ $(x,s)\in\Om\X(0,T)$, therefore $V(x,s)w(x,s)\ge0$, for a.e.\ $(x,s)\in\Om\X(0,T)$, by monotonicity. Using also 
\eqref{e6.8'}, we then obtain from \eqref{e6.18}
\begin{equation}\label{e6.19}
\int_{\Om}|\nabla A^{-1} w(x,t)|^2\,dx\le 0, \quad\text{for a.e.\ $t\in(0,T)$},
\end{equation}
and since $A^{-1}w\in L^2(0,T;H_0^1(\Om))$,  we get that $w=0$, a.e.\ in $\Om\X(0,T)$, that is $w_1=w_2$, and, hence, $V_1=V_2$, as desired.

\section*{Acknowledgements}

H.~Frid gratefully acknowledges the support from CNPq, through grant proc.~303950/2009-9, and FAPERJ, through grant E-26/103.019/2011.



\bibliographystyle{plain}
\bibliography{referenceFSV1}

\end{document}